\documentclass[10pt, a4paper]{amsart}
\usepackage[utf8]{inputenc}
\usepackage[T1]{fontenc}
\usepackage{paralist}
\usepackage{amssymb}
\usepackage{amsmath}
\usepackage{amsthm}
\usepackage[ngerman,english]{babel}%\usepackage{cite}
\usepackage[hang]{caption}
\usepackage{fancyhdr}
\usepackage{chngcntr}
\usepackage{tikz}
\newtheoremstyle{definition}
  {7pt}                                                                   %% space above
  {7pt}                                                                   %% space below
  {}                                                                      %% body font
  {}                                                                      %% indent amount
  {\normalfont \bfseries}                                                 %% head font
  {\normalfont}                                                           %% punctuation after head
  { }                                                                     %% space after head
  {{\normalfont \bfseries \thmname{#1}\thmnumber{ #2}\thmnote{ (#3)}}}    %% head spec

\newtheoremstyle{theorems}
  {10pt}                                                                  %% space above
  {10pt}                                                                  %% space below
  {\itshape}                                                              %% body font
  {}                                                                      %% indent amount
  {\normalfont \bfseries}                                                 %% head font
  {\normalfont}                                                           %% punctuation after head
  { }                                                                     %% space after head
  {{\normalfont \bfseries \thmname{#1}\thmnumber{ #2}\thmnote{ (#3)}}}    %% head spec

\theoremstyle{definition}

\newtheorem{mydef}[subsection]{Definition}

\newtheorem{remark}[subsection]{Remark}

\theoremstyle{theorems}

\newtheorem{theorem}[subsection]{Theorem}
\newtheorem*{theorem*}{Theorem.}
\newtheorem{prop}[subsection]{Proposition}
\newtheorem{lemma}[subsection]{Lemma}
\newtheorem{corollary}[subsection]{Corollary}

\newcounter{paranum}[section]
\renewcommand{\theparanum}{\thesection.\arabic{paranum}}
\newcommand{\Paragraph}{\vspace{10pt}\noindent \refstepcounter{paranum}\textbf{\theparanum}\textbf}

 \DeclareMathOperator{\closess}{cl_{ess}}

\newcommand{\Oplam}{\mbox{\Large $\curlywedge$}}

\newcommand{\DDD}{\mbox{D}}
\newcommand{\vol}{\mbox{Vol}}

\newcommand{\Hmm}[1]{\leavevmode{\marginpar{\tiny%
$\hbox to 0mm{\hspace*{-0.5mm}$\leftarrow$\hss}%
\vcenter{\vrule depth 0.1mm height 0.1mm width \the\marginparwidth}%
\hbox to 0mm{\hss$\rightarrow$\hspace*{-0.5mm}}$\\\relax\raggedright
#1}}}

\newcommand{\SSS}{\mathbb{S}}

\newcommand{\equi}{\ensuremath{\Leftrightarrow}}
\newcommand{\follows}{\ensuremath{\Rightarrow}}

\newcommand{\ld}{\ensuremath{,\ldots,}}
\newcommand{\ssq}{\ensuremath{\subseteq}}
\newcommand{\smin}{\ensuremath{\setminus}}
\newcommand{\eps}{\ensuremath{\varepsilon}}
\newcommand{\wh}{\ensuremath{\widehat}}

\newcommand{\htop}{\ensuremath{h_{\mathrm{top}}}}

\newcommand{\interior}{\ensuremath{\mathrm{int}}}

\newcommand{\closure}{\ensuremath{\mathrm{cl}}}

\newcommand{\nfolge}[1]{\ensuremath{(#1)_{n\in\mathbb{N}}}}
\newcommand{\kfolge}[1]{\ensuremath{(#1)_{k\in\mathbb{N}}}}

\newcommand{\jfolge}[1]{\ensuremath{(#1)_{j\in\mathbb{N}}}}

\newcommand{\twovector}[2]{\ensuremath{\left(\begin{array}{c} #1 \\
        #2 \end{array}\right)}}

\newcommand{\alphlist}{\begin{list}{(\alph{enumi})}{\usecounter{enumi}\setlength{\parsep}{2pt}
      \setlength{\itemsep}{1pt} \setlength{\topsep}{5pt}
      \setlength{\partopsep}{3pt}}}
\newcommand{\arablist}{\begin{list}{(\arabic{enumi})}{\usecounter{enumi}\setlength{\parsep}{2pt}
          \setlength{\itemsep}{1pt} \setlength{\topsep}{5pt}
          \setlength{\partopsep}{3pt}}}
\newcommand{\romanlist}{\begin{list}{(\roman{enumi})}{\usecounter{enumi}\setlength{\parsep}{2pt}
              \setlength{\itemsep}{1pt} \setlength{\topsep}{5pt}
              \setlength{\partopsep}{3pt}}}
\newcommand{\Romanlist}{\begin{list}{(\Roman{enumi})}{\usecounter{enumi}\setlength{\parsep}{2pt}
              \setlength{\itemsep}{1pt} \setlength{\topsep}{5pt}
              \setlength{\partopsep}{3pt}}}
\newcommand{\bulletlist}{\begin{list}{$\bullet$}{\setlength{\parsep}{2pt}
                \setlength{\itemsep}{1pt} \setlength{\topsep}{5pt}
                \setlength{\partopsep}{3pt}\setlength{\leftmargin}{15pt}}}
\newcommand{\Alphlist}{\begin{list}{(\Alph{enumi})}{\usecounter{enumi}\setlength{\parsep}{2pt}
      \setlength{\itemsep}{1pt} \setlength{\topsep}{5pt}
      \setlength{\partopsep}{3pt}}}
 \newcommand{\listend}{\end{list}}

\newcommand{\T}{\ensuremath{\mathbb{T}}}

\newcommand{\N}{\ensuremath{\mathbb{N}}}
\newcommand{\R}{\ensuremath{\mathbb{R}}}
\newcommand{\Z}{\ensuremath{\mathbb{Z}}}
\newcommand{\Q}{\ensuremath{\mathbb{Q}}}
\newcommand{\C}{\ensuremath{\mathbb{C}}}
\newcommand{\Proj}{\ensuremath{\mathbb{P}}}

\newcommand{\cD}{\mathcal{D}}

\newcommand{\cL}{\mathcal{L}}

\newcommand{\cN}{\mathcal{N}}

\newcommand{\cW}{\mathcal{W}}

\newcommand{\jLim}{\ensuremath{\lim_{j\rightarrow\infty}}}

\title[Model sets with positive entropy]{Model sets with positive entropy in Euclidean cut and project schemes}
\author{T.~J\"ager, D.~Lenz and C.~Oertel}

\begin{document}

 \maketitle

 \begin{abstract}
  We construct model sets arising from cut and project schemes in Euclidean
  spaces whose associated Delone dynamical systems have positive topological
  entropy. The construction works both with windows that are proper and with
  windows that have empty interior. In a probabilistic construction with
  randomly generated windows, the entropy almost surely turns out to be
  proportional to the measure of the boundary of the window.\medskip

\noindent\textsc{Resum\'e.}\ \ On construit des ensembles de Delone euclidiens
obtenus par coupe et projection de sorte que l'entropie du syst\`eme dynamique
associ\'e soit strictement positive. La construction permet d'utiliser une
fen\^etre propre ou d'int{\'e}rieur vide. Dans une construction probabiliste, pour
presque tout param\^etre, l'entropie est proportionnelle \`a la mesure de la
fronti\`ere de la fen\^etre.

 \end{abstract}

 \section{Introduction}

In the last decades, aperiodic order -- often referred to as the
mathematical theory of quasicrystals -- has developed into a broad
and highly active field of research, see e.g. \cite{BG,KLS} for
recent books dealing with this topic. In this context, the main
attention has been given to models with a strong degree of
long-range order. In particular, there is nowadays a fairly good
understanding of the relations between pure point diffraction --
characterising quasicrystals from the physical viewpoint -- and
purely discrete dynamical spectrum, which has emerged as one of the
major tools in the mathematical analysis of long-range aperiodic
order.

In this paper, we have a slightly different focus and construct
models that may be considered as intermediate between strong
long-range order and disorder. More precisely, we introduce a broad
family of model sets, produced by cut and project schemes in
Euclidean space, whose associated Delone dynamical systems exhibit a
high degree of chaoticity, including positive topological entropy.
At the same time, they still inherit a certain degree of long-range
order, which is built into the underlying cut and project scheme and
manifests itself in a non-vanishing discrete part of the dynamical
spectrum as well as in  minimality. Although we restrict here to
study the basic dynamical properties, we  hope that the constructed
models may be instrumental in understanding the transition from
quasicrystalline to amorphous configurations in solid matter. We
note  several recent works dealing with similar  model sets with
'thick boundary' of the window, based on a variety of different
methods \cite{BHS,BJL,HP,HR,KR}. The reader may take that as an
indication for the timeliness of the endeavor.

We will discuss more specifically how the present paper relates to other works
and contributes to the emerging general theory towards the end of this section,
after we have introduced the necessary notation. Here, we already note that - to
the best of our knowledge - it provides the first examples of model sets with
positive entropy based on Euclidean cut and project schemes.

\smallskip

A {\em cut and project scheme (CPS)} is a triple $(G,H,\cL)$
consisting of locally compact abelian groups  $G$,  called
\textit{direct space},  and $H$, called \textit{internal space}, and
a discrete co-compact subgroup ({\em lattice}) $\cL\ssq G\times H$
such that the canonical projection $\pi_G:G\times H\to G$ is
one-to-one and the canonical projection $\pi_H : G\times H\to H$ has
dense image. This framework goes back to Meyer's influential book
\cite{Mey} and has later been developed in \cite{Moo97,Moo00,
schlottmann:GeneralizedModelSetsAndDynamicalSystems}.  In this paper
we will always take  $G = \R^N$ and we will assume $H$ to be
$\sigma$-compact (i.e.\ a countable union of compact sets) and
metrizable. Our main application concerns the case $G = H = \R$. So,
the reader  may also well think  from the very beginning of $H$ as
just another Euclidean space $\R^M$ (where $M\neq N$ is possible).

Given a relatively compact subset $W\ssq H$, which is called a {\em
window} in this context, such a CPS produces a uniformly discrete
subset of $G$ via
\[
\Oplam(W) \ = \ \pi_G\left(\cL\cap (G\times W)\right).
\]

An alternative way to define $\Oplam(W)$ is to introduce the {\em
star-map}. Set $L:=\pi_G(\cL)$  and $L^*:=\pi_H(\cL)$. Then, the
star map $*:L\to L^*$ is  given by $\ell\mapsto\ell^*$, where
$\ell^*$ is uniquely defined by $(\ell,\ell^*)\in\cL$ due to the
injectivity of $\pi_{G | \cL}$. Then, we have
$$\Oplam(W)=\{\ell\in L\mid
\ell^*\in W\}.$$

If $W$ has non-empty interior, then $\Oplam(W)$ is called a {\em
model
  set}, in the general case  it is called a {\em weak model set}. We will be
  concerned with  model sets whose window has  a further 'smoothness'
  feature:  A window $W\ssq H$ is called \emph{proper} (or sometimes
   \emph{topologically regular}) if
   $$\closure(\interior(W)) = W.$$
The associated model set will then also be referred to as
\textit{proper model set}. Note that any proper window is compact.

A model set is always Delone set (see the next section for more detailed
definitions and a discussion of further facts concerning CPS and model sets).

 %Given $\Lambda,\Gamma\ssq \R$, we define the distance between these two sets as
 %$d(\Lambda, \Gamma) = \min \left\{ \frac{1}{\sqrt{2}}, \tilde{d}(\Lambda,
 %\Gamma) \right\}$.

Given a window $W\ssq H$ (which will mostly be compact in our considerations
below), we can associate a dynamical system to $\Oplam(W)$ by considering the
$\R^N$-action $(s,\Lambda)\mapsto \Lambda-s$ on the {\em hull} of
$\Oplam(W)$. This hull is given as $\Omega(\Oplam(W)) =
\closure(\{\Oplam(W)-s\mid s\in\R^N\}$, where the closure is taken in a suitable
topology (defined below).  The properties of this dynamical system depend
crucially on the boundary of the window $W$.

If $W$ is proper and the boundary of $W$ has Haar measure zero, then
the dynamical system $(\Omega(\Oplam(W)),\R^N)$ is (measurably)
isomorphic to the Kronecker flow on the torus $\T = (\R^N\times
H)/\cL$ defined by $\omega : \R^N\times\T\to\T,\ (s,\xi)\mapsto \xi
+ [s,0]_\cL$ and is therefore uniquely ergodic with purely discrete
dynamical spectrum
\cite{schlottmann:GeneralizedModelSetsAndDynamicalSystems} and zero
topological entropy \cite{BLR}. This case has attracted most
attention in recent years. In fact, it seems fair to say that
\textit{regular model sets}, i.e.\ sets of the form $\Oplam (W)$ for
proper $W$ whose boundary has measure zero, are the prime examples
for quasicrystals.  In particular, substantial efforts have been
spent over the years to prove pure point diffraction for regular
model sets, see e.g. \cite{Hof,
schlottmann:GeneralizedModelSetsAndDynamicalSystems}. By now this
pure pointedness is well understood and   three different approaches
have been developed: The approach of \cite{Hof} via Poisson
summation formula has recently been extended to a very general
framework in \cite{RS}. The result of
\cite{schlottmann:GeneralizedModelSetsAndDynamicalSystems} can be
seen within the context of the equivalence between purely discrete
dynamical spectrum and pure point diffraction, proven in this
setting in \cite{leeEtAl:PurePointDynamicalAndDiffractionSpectra}
and later generalized in various directions in e.g.
\cite{BL,Goue,LS,LM}. Finally,  pure point spectrum can also be
shown using  almost periodicity \cite{BM}, see also \cite{Stru}.

%The emergence of quasicrystals of this type in the temperature zero
%limit in physical models has been studied, for example, in
%\cite{BruinLeplaideur}.

Conversely, the case of windows with `thick boundary', in the sense of positive
Haar measure, is not as well understood. A general idea in this context is that
thickness of the boundary should imply positive topological entropy and failure
of unique ergodicity. In fact, corresponding conjectures have been brought
forward by Moody, see \cite{HR} for discussion, and Schlottmann
\cite{schlottmann:GeneralizedModelSetsAndDynamicalSystems}. These conjectures
are supported by prominent examples. Indeed, for the well-known example of
visible lattice points the associated dynamical system is far from being
uniquely ergodic and has positive topological entropy \cite{BMP,HP}. This system
has still pure point diffraction \cite{BMP} and pure point dynamical spectrum if
it is equipped with a natural ergodic measure \cite{HB}.  Existence of such a
canonical ergodic measure for general model sets with thick boundary has
received attention recently, see \cite{BHS} for an approach based on a maximal
density condition and \cite{KR} for an rather structural approach. Quite
remarkably, all these model sets with maximal density still have pure point
diffraction and pure point dynamical spectrum with respect to the canonical
measure \cite{BHS}. Note, however, that the eigenfunctions will in general not
be continuous anymore. In this context, a general upper bound on topological
entropy has been established in \cite{HR}. Given this support for the mentioned
conjectures, the recent findings in \cite{BJL} may seem surprising as they provide
examples of proper model sets with thick boundary which are still uniquely
ergodic (and minimal) with topological entropy zero. At the same time \cite{BJL}
also provides some examples of proper model sets with minimal dynamical systems
of positive entropy lacking unique ergodicity. All examples of \cite{BJL} are
based on Toeplitz systems.

In all  examples in the preceding discussion, where the topological
entropy was shown to be positive,  the internal space $H$ is not an
Euclidean space but has a rather more complicated structure (being a
$p$-adic space in the case of the visible lattice points and being
an odometer in the case of the Toeplitz systems). In the  present
paper we provide examples  of model sets with positive entropy based
on Euclidean internal space.

For the sake of simplicity, we will here restrict to Euclidean CPS with
one-dimensional internal space $H=\R$. In principle, similar constructions can
be carried out with higher-dimensional internal group, see
Section~\ref{HigherDimInternal} for a brief discussion.  Then, a lattice with
the above properties is of the form $\cL=A(\Z^{N+1})$, where
$A\in\mathrm{GL}(N+1,\R)$ satisfies the two conditions that $\pi_1:\R^{N+1}\to
\R^N$ is injective on $\cL$ and $\pi_2:\R^{N+1}\to \R$ maps $\cL$ to a dense
set. Note that this is certainly a generic condition on $A$ this is always
satisfied whenever the entries of $A$ are linearly independent over $\Q$.  We
call such $\cL$ an {\em irrational lattice}.
%In this situation there exists a further symmetry between $G$ and
%$H$ viz both canonical projections, denoted by $\pi_1$ and $\pi_2$
%are one-to-one with dense range.
The situation can be summarized in the following diagram.
 \begin{center}
  \begin{tabular}{ccccc}
   $\mathbb{R}^N$ & $\overset{\pi_1}{\longleftarrow}$ & $\mathbb{R}^N \times
  \mathbb{R}$ & $\overset{\pi_2}{\longrightarrow}$ & $\mathbb{R}$ \\
   $ \cup $ & & $\cup$ & & $\cup$ \\
   $L$ & $\overset{1-1 }{\longleftarrow} $ & $\cL = A(\mathbb{Z}^{N+1})$
   & $\overset{dense}{\longrightarrow}$ & $L^*$ \\
  \end{tabular}
  \end{center}

In this setting we construct  examples with positive topological
entropy (in fact, the maximal entropy possible given the bound in
\cite{HR}) and lack of unique ergodicity. At the same time these
examples still are minimal and have  a relatively dense set of
continuous eigenvalues. So, our examples share positive entropy and
lack of unique ergodicity with the examples of \cite{BMP,HP} while
they differ from these examples by having the additional regularity
feature of minimality and  a dense set of continuous eigenvalues. On
the other hand our examples share minimality and positive entropy
with the mentioned examples of \cite{BJL} but differ from these
examples by being based on a Euclidean CPS.
% Moreover, they have maximal possible entropy, whereas the considerations of \cite{BJL}
% do not offer   explicit control on the value of the entropy.

To us, a main achievement of our construction is that it is rather
direct and transparent. By this we hope that it can serve as a tool
for further investigations as well.

In order to to give a flavor of our results, we will next state one main theorem
(an extended version of which is given below in Theorem
\ref{t:random_window-extended}), which focuses a probabilistic model with
`random' window. Deterministic constructions are given as well, in Section
\ref{section:WindowsConsistingOfCantorSets} for the case of model sets and in
Section \ref{section:WindowsWithEmptyInterior} for the case of weak model
sets. The latter has started to attract increasing attention due to its
relations to number theory, compare discussion above and \cite{HR,BHS}.

\begin{theorem} \label{t.random_windows}
  Suppose $\cL\ssq \R^{N+1}$ is an irrational lattice and $C$ is a Cantor set of
  positive Lebesgue measure in $[0,1]$.  Let $(G_n)_{n\in\N}$ be a numbering of
  the bounded connected components of $\R\smin C$ and
  $\Sigma^+=\{0,1\}^\N$. Denote by $\Proj$ the Bernoulli distribution on
  $\Sigma^+$ with equal probability $1/2$ for each symbol and define
  \[
      W(\omega) \ = \ C \cup \bigcup_{n\in\N:\omega_n=1} G_n \ ,
  \]
where $\omega\in\Sigma^+$. Then for $\Proj$-almost every $\omega \in \Sigma^+$
the set $W(\omega)$ is proper and the dynamical system
$(\Omega(\Oplam(W(\omega)+\vartheta)),\R^N)$ has positive topological entropy
for all $\vartheta \in\R$ and is minimal for $\vartheta$ from a residual subset
$\Theta \ssq\R$ (depending on $\omega$).
\end{theorem}
\begin{remark} \label{rem:kommentar-hauptergebnis}
  \alphlist
 \item In fact, the topological entropy attains the   upper bound provided in \cite{HR},  which is
given in terms of the measure of $\partial W(\omega)$ and the
density of the lattice $\cL$, see Theorem
\ref{t:random_window-extended}.

\item The existence of the  residual subset $\Theta\ssq \R$ such that
  for all $\vartheta\in \Theta$ the system $(\Omega(\Oplam(W(\omega)+\vartheta)),\R^N)$
  is minimal is a consequence of  general (and well-known)  theory of model sets and has nothing to do with our (random) setting.
\item Due to the properness of the window our systems fibre  over a torus, i.e. allow for a torus as a factor.
This has some consequences:  For one thing, by abstract results this
then implies that the entropy comes from single fibres (see Remark
\ref{rem:Bowen} below). In fact, our proof directly exhibits fibres
carrying the entropy. Also, having this factor implies that our
examples have a relatively dense set of continuous eigenvalues, see
Remark \ref{rem:continuousef}.

\item Our results also show that if $|C|>1/2$, then for the set of $\omega$ of
  full measure above and any $\vartheta\in\R$ the dynamical system
  $(\Omega(\Oplam(W(\omega)+\vartheta)),\R)$ is not uniquely ergodic (see
  Theorem \ref{t:random_window-extended}).  \listend
\end{remark}

The reason for the positive toplogical entropy of
$(\Omega(\Oplam(W)+\vartheta),\R)$ is the existence of a large
'random component' in the hull, which may be of intrinsic conceptual
interest. We say $\Omega(\Oplam(W))$ {\em contains an  embedded
  fullshift}, if there exists $S \ssq \R^N$ of positive asymptotic density
  and a uniformly discrete $U\ssq \R^N$ such that for any subset $S'$ of $S$ there  exists
  $\Gamma\in\Omega(\Oplam(W))$ with $\Gamma \ssq U$ and
$$S' = \Gamma \cap S.$$
This means that we may think of the elements of $S$ as positions of
points (or atoms) which may be switched on or off completely
independently of each other, without leaving the hull (but there is
no control on what happens outside of $S$ at the same time). Details
are discussed in the first part of Section
\ref{section:EmbeddedSubshiftsandTopologicalIndependence}.  Embedded
fullshifts are closely related to the local structure of the window
$W$ (or its translate $W+\vartheta$) around the points in $L^*$. In
later parts of Section
\ref{section:EmbeddedSubshiftsandTopologicalIndependence}, we also
introduce the notion of local independence of $W$ with respect to
subsets of $L^*$ to establish criteria for the existence of embedded
fullshifts. Depending on the context, either a topological (Lemma
\ref{lem:localTopologicalIndependence}) or a metric version (Lemma
\ref{lem:localMetricIndependence}) of this concept can be applied. A
discussion of failure of unique ergodicity in the presence of
embedded subhifts is given in Section
\ref{section:FurtherPropertiesOfDDSWithPositiveTopologicalEntropy}.

The proof of Theorem~\ref{t.random_windows} is then given in Section
\ref{section:ProbabilisticallyWindowsAndPositiveEntropy}. In fact, Theorem
\ref{t:random_window-extended} in that section is an extended version of
Theorem~\ref{t.random_windows} including parts of Remarks
\ref{rem:kommentar-hauptergebnis}.
Section~\ref{section:WindowsConsistingOfCantorSets} then provides examples of
deterministic windows that equally lead to positive entropy. While this
construction is slightly more technical, it demonstrates that the randomness in
the definition of $W(\omega)$ above is not a key ingredient of the
procedure. Moreover, this also sets the ground for the construction of weak
model sets (whose window has empty interior) with positive entropy, which is
carried out in Section~\ref{section:WindowsWithEmptyInterior}.

\medskip

{\bf Acknowledgments.} The authors would like to thank an anonymous referee,
whose thoughtful remarks have led to substantial improvements of the paper.

TJ is supported by a Heisenberg grant of the German Research Council (DFG-grant
OE 538/6-1). Part of this work was done while DL was visiting the department of
mathematics at Geneva university. He would like to thank the department for its
hospitality.

 \section{Preliminaries}
  \label{section:CPSandDDS}
In this section we discuss the necessary background from the theory
of point sets and their associated dynamical systems. The material
is essentially well-known. For the convenience of the reader we
include some proofs.

\medskip

  \Paragraph{ Delone sets.}
     A set $\Lambda \ssq \mathbb{R}^N$ is called \emph{uniformly discrete} if
   there exists a real number $r > 0$ such that
    \begin{equation} \label{e.uniform_discreteness}
     \|x- y\| \ \geq r \ \quad \text{ for all } x,y \in \Lambda \
    \end{equation}
    where $\|\cdot\|$ denotes the Euclidean norm.
    The set is called \emph{relatively dense} if there exists a real number $R >
    0$ such that
    \begin{equation} \label{e.relative_denseness}
      B_R(x) \cap \Lambda \ \neq \  \emptyset  \text{ for all } x \in \mathbb{R}^N
   \ ,
    \end{equation}
    where $B_R (x)$ denotes the closed ball of radius $R$ around $x$.
    We call $\Lambda$ a \emph{Delone set} if it is uniformly discrete and
    relatively dense in $\mathbb{R}^N$. We say $p \in \mathbb{R}^N$ is a
    \emph{period} of $\Lambda$ if $\Lambda + p = \Lambda$ and call $\Lambda$
    \emph{aperiodic} if $p=0$ is the only period.    Given a Delone set $\Lambda$, let $x \in \Lambda$ and $\varrho > 0$. Then the pair
   $(P(\varrho,x),\varrho)$ with
    $$P(\varrho,x) := (\Lambda -x) \cap B_\varrho(0) $$ is called a \emph{$\varrho$-patch of
      $\Lambda$ in $x$}. The \emph{set of all patches} is given by
    $$\mathcal{P}(\Lambda) = \{ (P(\varrho,x),\varrho) \mid x \in \Lambda, \varrho > 0 \}. $$ Note
   that this definition works also for discrete sets which are not Delone.  The
   set $\Lambda$ has \emph{finite local complexity} (or \emph{(FLC)} for short) if
   \begin{equation}
    \sharp \{ (\Lambda - x) \cap B_\varrho(0) \mid x \in \Lambda \} < \infty \tag{FLC}
   \end{equation}
   for all $\varrho > 0$. This assumption is equivalent to various other properties:
  \begin{lemma}[\cite{lagarias:GeometricModelsForQuasicrystalsI}]
   \label{lem:DeloneFLCEquivalences}
   Let $\Lambda$ be a Delone set. Then the following statements are equivalent:
   \begin{enumerate}[(i)]
    \item $\Lambda$ has (FLC);
    \item $\sharp \{ (\Lambda - x) \cap B_{2R}(0) \mid x \in \Lambda \} <
      \infty$, where $R$ is as in (\ref{e.relative_denseness});
    \item $\Lambda - \Lambda$ is closed and discrete.
   \end{enumerate}
  \end{lemma}
  If $\Lambda$ is a Delone set with $\Lambda - \Lambda$  uniformly discrete, then $\Lambda$ is called a
  \emph{Meyer set}. Being a Meyer set is a notably strong property, and in
  particular implies (FLC) by Lemma \ref{lem:DeloneFLCEquivalences}
  (iii).

A Delone set $\Lambda$ is \emph{repetitive} if for all $(P,\varrho)
\in \mathcal{P}(\Lambda)$ the set $$ \{ x \in \Lambda \mid
P(\varrho,x) = P \}
$$ is relatively dense in $\mathbb{R}^N$. It has \emph{uniform patch
frequencies} (or \emph{(UPF)} for short) if for all patches $(P, \varrho) \in
\mathcal{P}(\Lambda)$ the limit
   \begin{equation} \label{e.UPF}
     \nu(P,x)\  = \  \lim_{n \rightarrow \infty}
   \frac{  \sharp \{ y \in (\Lambda-x) \cap B_n(0) \mid P(\varrho,y) = P \} }{\lambda (B_n(0))} \tag{UPF}
   \end{equation}
   exists and the convergence is uniform in $x \in \mathbb{R}^N$.  Here,
   $\lambda$ denotes the $N$-dimensional Lebesgue measure.

   \Paragraph{ Cut and project schemes and model sets.} In this
   section we discuss how Meyer sets arise from CPS. The material of
   this section is well-known \cite{Mey,Moo97,Moo00,schlottmann:GeneralizedModelSetsAndDynamicalSystems}.
   For the convenience of the reader we include some details and
   provide precise references.

\smallskip

We adopt the notation  introduced  in the introduction above and
consider a CPS $(G,H,\cL)$ with $G = \R^N$ and $H$ a locally compact
abelian group. We will assume that $H$ is $\sigma$-compact and
metrizable.\footnote{Metrizability of $H$ is only  a matter of
convenience. It allows us to work with sequences instead of nets. It
is clearly met in our specific examples, where we have $G = H =
\R$.} As both $\R^N$ and $H$ are  $\sigma$-compact, the lattice
$\cL$ must be countable (as it has a compact quotient).  The Haar
measure of a measurable subset $W\ssq H$ will be denoted by
$|W|$.\label{page:cpsintroduced}

\medskip

Here are the basic properties of sets arising from the CPS.

  \begin{lemma}[{\cite[Proof of Proposition 2.6 (i)]{Moo97}}]
   \label{lem:CPSMeyer}
   Let $(\R^N,H,\cL)$ be a CPS and $W\ssq H$.
Then, the following holds:
\begin{itemize}
\item $\Oplam (W)$ is uniformly discrete if  $\closure(W)$ is
compact.
\item $\Oplam (W)$ is relatively dense if $\interior
(W)\neq \emptyset$.
\end{itemize}
In particular,  $\Oplam(W)$ is a Delone set with  (FLC) (and even
Meyer)  if $W$ is relatively compact with non-emtpy interior.
  \end{lemma}
  \begin{proof} The statement of Proposition 2.6 (i) in \cite{Moo97}
  deals simultaneously  with both uniform discreteness and relatively denseness.
  However, the proof clearly gives both parts of the present lemma separately. Here, we only discuss how the last statement follows from the
  first two statements:    We have
   $$ \Oplam(W)-\Oplam(W) = \{x-y \mid  x^*, y^* \in W \}\ssq \{ z \in
L \mid z^* \in W-W \} = \Oplam(\closure(W) - \closure(W)).$$
   Since $\closure(W)- \closure(W)$ is  compact, $\Oplam( \closure(W)- \closure(W))$
    is uniformly discrete. Thus,
   also $\Oplam(W)-\Oplam(W)$ is uniformly discrete.
  \end{proof}

If $L^* \cap \partial W = \emptyset$, the model set is called
\emph{generic}. Here is the fundamental result on proper windows and generic
model sets. The result is a consequence of the Baire category theorem.

\begin{lemma}[{\cite[Proof of Corollary 4.4]{schlottmann:GeneralizedModelSetsAndDynamicalSystems}}]
\label{lem:properandgeneric} Let $(\R^N,H,\cL)$ be a CPS and $W\ssq
H$. If $\partial W$ has empty interior, then there exists an $h\in
H$ such that $W + h$ is generic. In particular, whenever $W$ is
proper there exists $h\in H$ such that $W + h$ is generic.
\end{lemma}
\begin{proof}  As $\cL$ is countable and $\partial W$ has empty interior,
$$L^* - \partial W = \bigcup_{l\in L} (l^* - \partial W)$$
can not agree with $H$ by Baire's category theorem.  Now, any $h\in H
\setminus ( L^* -\partial W)$ will have the desired property.

The last statement follows as for any proper window $W$, clearly,
its boundary $\partial W = W \setminus \interior (W)$ has empty
interior.
\end{proof}

A model set $\Oplam(W)$ is called \emph{regular} if $|\partial
W|=0$.

   \begin{lemma}[{\cite[Theorem 4.5]{schlottmann:GeneralizedModelSetsAndDynamicalSystems}, \cite[Theorem8]{Moo00}}]
       \label{lem:genericrepetitivity}
     \begin{itemize}
     \item[(a)] Let $\Oplam(W)$ be a regular model set associated to the  CPS
$(\R^N,H,\cL)$. Then it has (UPF).
     \item[(b)] Let $\Oplam(W)$ be a generic model set associated to the CPS $(\R^N,H,\cL)$. Then it is
repetitive.
     \end{itemize}
   \end{lemma}

  \Paragraph{ Delone Dynamical Systems.}
In this section we show how a uniformly discrete set gives rise to a
dynamical system. The dynamical systems arising in this way from
Meyer sets are the main object of study in our paper.

\smallskip

Let $\mathcal{F}$ denote the space of all closed subsets of $\R^N$
including the empty set.
%\Hmm{Wenn wir weak model sets behandeln wollen, muessen wir auch die leere Menge zulassen!}
Let furthermore $\mathcal{U}_{r}(\mathbb{R}^N)$ be the space of all uniformly
discrete sets in $\R^N$ which satisfy (\ref{e.uniform_discreteness}) with a
fixed constant $r>0$, and $\mathcal{D}_{r,R}$ be the set of all Delone sets with
satisfying (\ref{e.uniform_discreteness}) and (\ref{e.relative_denseness}) with
fixed constants $r,R>0$.  We can introduce a metric $d$ on $\mathcal{F}$ as
follows: Let
\[
    j \! : \, \SSS^N \xrightarrow{\quad}
          \R^{N}\cup\{\infty\}
\]
be the stereographic projection. Here, $\SSS^N$ denotes the
$N$-dimensional sphere in $\R^{N+1}$ and the point $\infty$ denotes
the additional point in the one-point compactification of $\R^{N}$,
which is the image of the `north pole' under $j$. Let
$d^{}_{\mathrm{H}}$ be the Hausdorff metric on the set of compact
subsets of $\SSS^N$. Then, for any closed  $\Lambda \ssq \R^{N}$,
the set $j^{-1} (\Lambda \cup\{\infty\}) $ is a closed and hence
compact subset of $\SSS^N$. Thus, via
\[
    d (\Lambda_1,\Lambda_2) \, := \,
      d^{}_{\mathrm{H}} \bigl( j^{-1} (\Lambda_1 \cup\{\infty\}),
          j^{-1} (\Lambda_2 \cup\{\infty\}) \bigr),
\]
we obtain a topology on the set of all closed subsets of $\R^N$.
%  For $\Lambda, \Gamma \in \mathcal{D}$ let
%   \[
%   \tilde{d}(\Lambda, \Gamma) \ = \ \inf \left\{ \varepsilon > 0 \mid \exists s
%   \in B_\varepsilon(0) : \Lambda \cap B_{\frac{1}{\varepsilon}}(0) = (\Gamma
%  -s) \cap B_{\frac{1}{\varepsilon}}(0) \right\} \ .
%  \]
%  and
%  $$d(\Lambda, \Gamma) = \min \left\{ \frac{1}{\sqrt{2}}, \tilde{d}(\Lambda, \Gamma)
%  \right\}.$$

  \begin{lemma}[\cite{LS}]
   \label{lem:DeloneMetric}
   The map $d: \mathcal{F} \times \mathcal{F}\to\R^+$ defines a metric on $\mathcal{F}$,
    which makes $(\mathcal{F},d)$
  into  a compact metric space. Further, the sets $\mathcal{U}_{r}$ and $\cD_{r,R}$ are compact in this
   metric for all $r,R>0$.
  \end{lemma}
\begin{proof} Compactness of $(\mathcal{F},d)$ is discussed in
\cite{LS}. As $\mathcal{U}_{r}$ and $\cD_{r,R}$ are clearly closed,
they are also compact.
\end{proof}

\begin{remark} In the investigation of  Delone sets (rather than uniformly discrete sets)
 another metric  may be even more common,  see e.g.
\cite{leeEtAl:PurePointDynamicalAndDiffractionSpectra}.  However,
both metrics induce the same topology, \cite{BL,LS}.
\end{remark}

   Let $\Lambda \ssq \mathbb{R}^N$ be a uniformly discrete  set. Then
    \[
    \Omega(\Lambda) \ = \ \closure \left( \left\{ \Lambda - s \mid s \in
        \mathbb{R}^N \right\} \right)
    \]
    is called the \emph{dynamical hull} of $\Lambda$. Here, the closure is taken
    with respect to the topology induced by the  metric discussed in  Lemma \ref{lem:DeloneMetric}.
Note that this closure may contain the empty set even if $\Lambda$
was not the empty set.
   Given the canonical flow $\varphi_s(\Gamma) := \Gamma - s$ on
    $\Omega(\Lambda)$, we call the pair $(\Omega(\Lambda), \varphi)$
    \emph{point set  dynamical system} and also write $(\Omega(\Lambda),
    \mathbb{R}^N)$.  Dynamical systems of this form are
    sometimes called \emph{mathematical quasicrystals}.

\begin{lemma}[{\cite[Corollary 3.3 and Proposition
3.1]{schlottmann:GeneralizedModelSetsAndDynamicalSystems}}]
] %%Corollary 3.3
  Let $\Lambda$ be a Delone set with FLC. Then \label{l.basic_dynamical_properties}
  \begin{itemize}
  \item[(a)] $(\Omega(\Lambda), \varphi)$ is uniquely ergodic if and only if
    $\Lambda$ has (UPF);
  \item[(b)] $(\Omega(\Lambda), \varphi)$ is minimal if and only if $\Lambda$ is
    repetitive.
  \end{itemize}
    \end{lemma}

\begin{remark} Let us note that the equivalence between minimality and a (suitably
defined) notion of repetitivity is true in much greater generality
as has been known since  \cite{Auslander}.
\end{remark}

Note that (FLC) is always fulfilled for model sets (see Lemma
\ref{lem:CPSMeyer} above).

\smallskip

The statement of the following proposition is  known and discussed
within proofs in
\cite{schlottmann:GeneralizedModelSetsAndDynamicalSystems,
baakeEtAl:CharacterizationOfModelSets}.

\begin{prop}\label{prop:eineroderallepunkte} Let $(\R^N,H,\cL)$ be a CPS and, as usual, $L = \pi_G (\cL)$.
Let  $\Lambda$ a Delone set
in $\R^N$ with $\Lambda\ssq L$. Then, for $\Gamma\in\Omega
(\Lambda)$ the following assertions are equivalent:

\begin{itemize}
\item[(i)] $\Gamma\ssq L$.
\item[(ii)] $\Gamma$ contains one point of $L$.
\end{itemize}

In this case, there exists a sequence $(t_n)$ in $L$ with $\Lambda +
t_n \to \Gamma$.
\end{prop}
\begin{proof} (i)$\Longrightarrow $ (ii): This is clear.

\smallskip

(ii)$\Longrightarrow$ (i): Let $x\in \Gamma \cap L$ be given.  Consider a
sequence $(t_n)$ in $\R^N$ with $\Gamma_n := \Lambda + t_n\to \Gamma$. Without
loss of generality we can then assume $x\in \Gamma_n$ for all $n\in\N$. We then
have $x\in L$ as well as $x \in L + t_n$ and this implies $t_n\in L$ for all
$n\in\N$. This gives, in particular, $\Gamma_n\ssq L$ for all $n\in\N$. Consider
now an arbitrary point $y\in\Gamma$. As $\Gamma_n\to \Gamma$ and
$x\in\Gamma_n,\Gamma$, we infer by finite local complexity that $y\in\Gamma_n$
for all sufficiently large $n$. This then implies $y\in L$.

\smallskip

The last statement has been proven along the proof of
(ii)$\Longrightarrow $ (i).
\end{proof}

\Paragraph{ Flow morphism and torus parametrisation.} The dynamical hull of a
Delone set arising from a CPS can be described via the so-called {\em torus
  parametrization}. This is discussed in this section.

\smallskip

Consider the CPS $(\R^N,H,\cL)$ and define the associated \textit{torus} by $$
\mathbb{T} := (\mathbb{R}^N\times H) / \cL.$$ Then $\mathbb{T}$ inherits a
natural group structure from $\R^N\times H$. We will write $[s,h]_\cL$ for the
element $(s,h) + \cL \in \T$. Further, there is a natural $\mathbb{R}^N$-action
on $\mathbb{T}$ given by
   $$\omega_s( \xi) := \xi + [s,0]_\cL.$$
For $s\in \R^N$ and $l\in L$, we then find
$$\omega_{s - l} (\xi) = \xi + [s-l,0]_\cL= \xi + [s,l^\ast]_\cL.$$ By the
denseness of $L^*$ in $H$, this shows that the action is minimal, i.e. each
orbit is dense. As $\T$ is a group, this gives that the action is uniquely
ergodic, i.e. there is only one invariant probability measure (see
\cite{schlottmann:GeneralizedModelSetsAndDynamicalSystems}).

   A {\em flow morphism} or \textit{factor map} between $\R^N$-actions
   $(X,\phi)$ and $(Y,\psi)$ is a continuous onto map $\eta:X\to Y$ which
   satisfies $\eta(\phi_s(x))=\omega_s(\eta(x))$ for all $x\in X$ and
   $s\in\R^N$. If such a flow morphism exists, the dynamical system $(Y,\psi)$
   is called a \textit{factor} of $(X,\phi)$.

\begin{prop}[\cite{baakeEtAl:CharacterizationOfModelSets}]\label{prop:DDSFlowMorphism}
  Let a  CPS $(\R^N H,\cL)$,  a proper window  $W\ssq H$ and   $\Lambda\ssq \R^N $
  with
    $\Oplam(\interior(W)) \ \ssq \ \Lambda \ \ssq \
    \Oplam(W)$ be given.
   Then there exists a unique  flow morphism $\beta : \Omega(\Lambda) \rightarrow
   \mathbb{T}$ with  $\beta(\Lambda) = 0$. This flow morphism
   satisfies
  \begin{equation}\label{eq:DDSFlowMorphism}
    \beta(\Gamma)  \ = \   [s,h]_{\cL} \quad \Longleftrightarrow \quad
    \Oplam(\interior(W)+h)-s
    \ \ssq \ \Gamma \  \ssq \ \Oplam(W + h)-s
   \end{equation}
   for $\Gamma \in \Omega(\Lambda)$.
\end{prop}

The  map $\beta$  from the previous proposition   is often  called a
   \emph{torus parametrization} (associated to the CPS $(\R^N,H,\cL)$
   and the window $W$) and this is how we will refer to it
   in the remainder of the paper. Note that it satisfies $\beta
   (\Lambda) = 0$.

\begin{remark}[Torus parametrization and continuous eigenfunctions]
\label{rem:continuousef} Existence of a torus parametrization has
   consequences for existence of  continuous eigenfunctions. Indeed,
   in the situation of the preceding proposition we can define  for  any $\gamma$ in
   the dual group of $\T$, i.e. any continuous group homomorphism
   $\gamma : \T\longrightarrow\{ z\in \C: |z|=1\}=:S^1$,  the function
  $f := f_\gamma := \gamma \circ \beta$ on $\Omega(\Lambda)$. This
function satisfies
$$f(\phi_s (\Gamma)) = \gamma ( (\beta (\Gamma)) + [s,0]_\cL) = \gamma
([s,0]_\cL) f(\Gamma) = \gamma^* (s) f(\Gamma) $$ for all $s\in
\R^N$ and $\Gamma\in \Omega(\Lambda)$, where we have defined
$\gamma^* : \R^N\longrightarrow S^1$, by $\gamma^* (s) =\gamma
([s,0])$. Then, $\gamma^* $ is an element of the dual group of
$\R^N$ and, hence, $f$ is a continuous eigenfunction. Moreover, the
general theory of CPS, as discussed  in \cite{Mey,Moo97}, shows that
the set $\{\gamma^* : \gamma \in \mbox{dual group of $\T$}\}$ is
relatively dense in $\R^N$. So, we have a relatively dense set of
eigenvalues with continuous eigenfunctions.
%Indeed, within the
%minimal point set dynamical systems of finite local complexity this
%is a characteristic feature of systems consisting of Meyer sets, see
%\cite{KS}.
\end{remark}

%  \begin{remark}\label{r.flow_morphism}
%    It is possible to extend the flow morphism $\beta$ to the closure of
%    $\{\Oplam(W+\vartheta)\mid \vartheta\in H\}$, which contains all hulls
%    $\Omega(\Oplam(W+\vartheta))$. Then we have the simple relation
%    \[
%      \beta_\vartheta(\Gamma)\ =\ \beta(\Gamma)-[0,\vartheta]_{\cL} \ .
%    \]
%  \end{remark}

The structure  of fibres of $\beta$ will be crucial for our further
investigation. The following lemma  underlines the spirit of the
constructions of the next section. Similar arguments can be found
e.g.  in \cite{baakeEtAl:CharacterizationOfModelSets}.

  \begin{lemma}
   \label{lem:subshiftConvergence}
   Let $(\R^N, H,\cL)$ be a CPS and  $W\ssq H$ be a proper window, $\Lambda = \Oplam (W)$  and $\beta$
   the associated torus parametrization.
   For given $[0,h]_{\cL} \in \mathbb{T}$, the following conditions are equivalent:
   \begin{enumerate}[(i)]
    \item $\Gamma \in \beta^{-1}([0,h]_{\cL})$;
    \item There exists a sequence $h_j \in L^*$ such that $\lim_{j \rightarrow \infty} h_j = h$ and
     $$\lim_{j \rightarrow \infty} \Oplam(W+h_j) = \Gamma.$$
   \end{enumerate}
  \end{lemma}
  \begin{proof} (i)\follows(ii):
    By $\beta (\Gamma) = [0,h]_\cL$ and Proposition
    \ref{prop:DDSFlowMorphism} we  have
    $$\Gamma \ssq \Oplam
    (W-h) \ssq L.$$ Now,  from Proposition
    \ref{prop:eineroderallepunkte} we obtain a sequence $s_j\in L$
    with $\Lambda -s_j\to \Gamma$.  Due to the
    continuity of the flow morphism $\beta$, we then  obtain
\begin{eqnarray*}   [0,h]_{\cL}  =  \beta(\Gamma) =   \lim_{j \rightarrow
    \infty} \beta(\varphi_{s_j}(\Oplam(W)))  =  \lim_{j \rightarrow
    \infty} [0,s_j^*]_{\cL}.
\end{eqnarray*}
This easily implies convergence of $h_j:= s_j^*$ to $h\in H$ for
$j\to \infty$.
%By slightly shifting the $s_j$, we may therefore assume that
%%$[s_j,0]_\cL=[0,t_j]_\cL$ for a sequence of $t_j\in\R$ with $\jLim t_j=t$. Then
%\[
%\Oplam(W+t_j) \ = \ \Oplam(W)-s_j \
%\stackrel{j\to\infty}{\longrightarrow} \ \Gamma
%\]
%as required.

(ii)\follows(i): This follows   immediately from the continuity of
$\beta$.
%: if $\jLim \Oplam(W+t_j)=\Gamma$ and $\jLim t_j=t$, then
%$\beta(\Gamma)=\jLim[0,t_j]_\cL=[0,t]_\cL$.
  \end{proof}

%  Thus, \eqref{e.subshift}
%  in condition $(ii)$ of Definition \ref{def:DSSubshift} is
%  equivalent to the existence of a sequence of $ l_j^* \in L^* \text{ with }
%  \lim_{j \rightarrow \infty} l_j^* = t$ such that for all $j \in \mathbb{N}$,
%  $i=1, \dots, K$ the following equivalence holds:
%  \begin{equation}
%   \label{eq:LocalIndependenceDefinitionOfSequence2}
%   l^*_{n_i} \in W + l^*_j \ \Longleftrightarrow \  a_i = 1.
%  \end{equation}

%  \begin{remark}
%   \label{rem:LocalIndependenceRemarkEquivalences}
%   Note that instead of $\{ l_j^* \} \ssq L^*$ and $\lim_{j \rightarrow
%     \infty} l_j^* = t$ we can also write
%   \begin{equation}
%     \label{eq:LocalIndependenceDefinitionOfSequence1'}
%    \tilde{l}_j^*   :=  l_j^*-t  \in  L^*-t \quad \text{and} \quad \lim_{j
%      \rightarrow \infty} \tilde{l}_j^*  =  0 \ ,
%   \end{equation}
%   and then (\ref{eq:LocalIndependenceDefinitionOfSequence2}) is equivalent to
%    \begin{equation}
%      \tag{\ref{eq:LocalIndependenceDefinitionOfSequence2}'} \label{eq:LocalIndependenceDefinitionOfSequence2'}
%      l^*_{n_i} \in W + t + \tilde{l}^*_j \ \Longleftrightarrow \ a_i = 1.
%   \end{equation}
%  \end{remark}

%  At the same time, we can write the left side in
%  (\ref{eq:LocalIndependenceDefinitionOfSequence2}) as
%   $$l^*_{n_i} - l^*_j \in W. $$ This shows that the problem of finding an
%   embedded fullshift is actually related to analyzing the local structure of
%   the window $W$ in some neighbourhood of the points $l^*_{n_i}$.

From the considerations in
\cite{baakeEtAl:CharacterizationOfModelSets} we obtain the following
lemma. As we will need to build on this argument in later sections
we include a short proof.

\begin{lemma}[\cite{baakeEtAl:CharacterizationOfModelSets}] \label{lem:dichotomy}%% Prop. 8
    Let $(\R^N,H,\cL)$ be a CPS and $W\ssq H$ a proper
    window and
   $\Lambda \ssq \R^N$ with $
    \Oplam(\interior(W)) \ \ssq \ \Lambda \ \ssq \
    \Oplam(W)$ be given.
    Then the   following dichotomy  holds for the torus parametrization:
\begin{itemize}
\item[(a)]   If $\emptyset  =  (\partial W + h) \cap L^*$  then $[0,h]_{\cL}$ has
exactly one preimage under $\beta$.
\item [(b)] If there exist an $l\in L$ with $l^* \in   \partial W+h$, then
   $\beta^{-1}([0,h]_{\cL})$ contains at least two elements $\Gamma$ and
   $\Gamma'$ which satisfy $l \in \Gamma$ and $l \notin \Gamma'$.
\end{itemize}
In particular,  $[0,h]_{\cL}$ has exactly one preimage under $\beta$
if and only if $W+h$ is generic i.e. $(\partial W + h)  \cap L^* =
\emptyset$ holds. Moreover,   there exists an $h\in H$ such that the
fibre $\beta^{-1} ([0,h]_\cL)$ has only one element.
\end{lemma}

  \begin{proof} Clearly,  $(W +h) \cap L^*  = (\interior(W) + h)
\cap L^*$ if and only if $(\partial W +h) \cap L^* = \emptyset$.

\smallskip

Consider first the case $(\interior(W)+h) \cap L^* = (W+h) \cap
L^*$. Then, $[0,h]_{\cL}$ has exactly one preimage under $\beta$ by
$\eqref{eq:DDSFlowMorphism}$.

Consider now the case $l^* \in \partial W+h$ for some $l^* \in L^*$.
Since $L^*$ is dense in $H$ and $W$ is proper, we can find
    elements $s_n, s'_n\in L$, $n\in \mathbb{N}$,  such that $h_n = s_n^*$  and
$h'_n = (s'_n)^*$ satisfy
   \begin{itemize}
    \item $\lim_{n \rightarrow \infty} h_n = \lim_{n \rightarrow \infty} h_n' = h$,
    \item $l^* \in \interior(W)+h_n$ and $l^* \notin W+h_n'$ for all $n \in \mathbb{N}$.
   \end{itemize}
   By going over to subsequences if necessary, we may assume that $\{
   \varphi_{-s_n}(\Lambda)\}_{n \in \mathbb{N}}$ and $\{
   \varphi_{-s_n'}(\Lambda) \}_{n \in \mathbb{N}}$ converge to some elements
   $\Gamma$ and $\Gamma'$ of the hull $\Omega(\Oplam(W))$, respectively. Since
   $$\varphi_{-s_n}(\Lambda ) =    \Lambda + s_n \supset  \Oplam(\interior(W)) +  s_n
    = \Oplam(\interior(W) +h_n)\ni l,$$
    we obtain $l \in \Gamma$. In a
   similar way, we can show that at the same time $l \notin \Gamma'$. Hence, we
   obtain that $\Gamma \neq \Gamma'$.  As $\beta$ is a flow morphism, we have
   \begin{align*}
     \beta(\Gamma) &= \beta\left(\lim_{n \rightarrow \infty}
       \varphi_{-s_n}(\Lambda )\right) = \lim_{n \rightarrow \infty}
     \beta(\varphi_{-s_n}(\Lambda )) = \lim_{n \rightarrow \infty} \omega_{-s_n}(\beta(\Lambda)) \\
     &= \lim_{n \rightarrow \infty} \omega_{-s_n}(0) = \lim_{n \rightarrow
       \infty} [-s_n,0]_{\cL} = \lim_{n \rightarrow \infty} [0,h_n]_{\cL} =
     [0,h]_{\cL}.
   \end{align*}
   The same holds for $\Gamma'$, and hence $\Gamma, \Gamma' \in
   \beta^{-1}([0,h]_{\cL})$.

\smallskip

The 'In particular' statement is  immediate from the preceding two
statements. The last statement is then clear  by  Lemma
\ref{lem:properandgeneric}.
\end{proof}

%\begin{remark}\label{rem:mef} The corollary gives that for proper
%$W$ the dynamical system $(\Omega(\Lambda),\R^N)$ is
%almost-automorphic and  $\T$ is its  maximal equicontinuous factor.
%As we do not use this here, we refrain from further discussion but
%rather refer the reader to \cite{ABKL} for a recent survey on this
%topic and \cite{Auj} for a recent characterization of Meyer
%dynamical systems which are almost-automorphic.
%\end{remark}

From Lemma \ref{lem:dichotomy} and the fact that $L^*$ is countable, we also
immediately obtain that regularity of the window $W$ has strong implications for
the fibre structure. To state this more precisely we need the following piece of
notation: Two measure-preserving $\R^N$-actions $(X,\phi,\mu)$ and
$(Y,\psi,\nu)$ are {\em measure-theoretically isomorphic} if there exist full
measure sets $X_0\ssq X$ and $Y_0\ssq Y$ and a measurable bijection $\eta:X_0\to
Y_0$ such that $\eta\circ \phi_s(x)=\psi_s\circ\eta (x)$ for all $x\in X_0$ and
$s\in\R^N$.

  \begin{corollary}[\cite{baakeEtAl:CharacterizationOfModelSets}] \label{cor:metricisom}%%Theorem 5
    Consider  the situation of the previous lemma and assume $|\partial W| = 0$. Then
    for $\lambda$-almost all $\xi \in \mathbb{T}$ the preimage
    $\beta^{-1}(\xi)$ is a singleton. \label{l.singleton_fibres} In
    particular, the flow $(\Omega(\Lambda), \varphi)$
  is uniquely ergodic and measure-theoretically isomorphic to $(\mathbb{T},
  \omega)$ .  \end{corollary}

  \begin{remark}
   \label{rem:DDSWeakHulls}
   Whenever we consider  Delone dynamical systems which arise from proper model
   sets, the preceding results  on the torus parametrization form a basis for  our treatment.
However, we will also consider the dynamical hull of weak model
sets, which are not proper. In this case, we cannot appeal to the
previous  results. In fact, if $W$ is compact with
$\interior(W)=\emptyset$ the hull of $\Oplam (W)$ must  contain the
empty set and  there can not  exist a torus parametrization as the
empty set is fixed by the action, whereas no point of the torus is
fixed by the action.
  \end{remark}

  In the sequel, we will often deal with proper windows $W$ and $\Lambda =
  \Oplam (W)$. We will then also need to replace the window $W$ by any of its
  translates $W+\vartheta$, $\vartheta\in H$. In this case, Proposition
  \ref{prop:DDSFlowMorphism} (applied to $W + \vartheta$ instead of $W$) yields
  a unique flow morphism
\begin{equation}\label{r.flow_morphism}
 \beta_\vartheta: (\Omega(\Oplam(W+\vartheta)),\varphi) \longrightarrow (\T,\omega),
\end{equation}
which sends     $\Oplam(W+\vartheta)$ to $0$.  For $\vartheta =0\in
H$ we will still write $\beta$ instead of  $\beta_0$.

\Paragraph{ Uniform distribution and asymptotic densities.}
Densities of subsets of Euclidean space will play an important role
in our considerations. Here, we discuss the necessary tools.

\smallskip

In the following, we use the partial
  ordering on $\R^N$ which is given by $s\leq t \equi s_i\leq t_i$ for all
  $i=1\ld N$. Given $t\in\R$, we let $\bar t=(t\ld t)\in\R^N$. Thus

  $$F_t:=\{s\in\R^N\mid -\bar t\leq s \leq \bar t\}$$
  is a cube of sidelength $2t$
  and volume $(2t)^N$.

\smallskip

Whenever $S$ is a uniformly discrete  subset of $\R^N$ we  define
its \textit{asymptotic density} by
$$\nu_S := \limsup_{t\to \infty} \frac{ \sharp S\cap F_t}{\lambda(F_t)},$$ where
$\sharp A$ denotes the cardinality of $A$.  If the limsup is actually a limit,
we call it the \textit{density} of the set $S$.  Model sets provide an instance
where densities tend to exist rather generally. This is sometimes discussed
under the header 'uniform distribution'. In order to state the corresponding
result we will need one more piece of notation: Let $(\R^N,H,\cL)$ be a
CPS. Then a subset $\cD$ of $\R^N\times H$ is called a \textit{fundamental domain}
of $(\R^N \times H) / \cL$ if it contains exactly one representative of any
element in the quotient group. As is well-known, the volume of a (measurable)
fundamental domain does not depend on the choice of the actual fundamental
domain. We denote this volume by $\vol(\cL)$. Note that in the Euclidean case
$\vol(\cL)=\det(A)$. We can now recall a result from \cite{Moody}, which in our
setting gives the following.

\begin{theorem}[Uniform distribution for model
sets \cite{Moody}]\label{t:uniform-distribution} Let $(\R^N,H,\cL)$
be a CPS and $W\ssq H$ measurable. Then, the following holds:

(a) For  almost every $\vartheta\in H$ (with respect to Haar measure
on $H$) the density of $\Oplam (W +\vartheta)$ exists and is given
by $ \frac{|W|}{\vol(\cL)}$.

(b) If $W$ is compact,  the inequality
$$\limsup_{t\to \infty} \frac{ \sharp \Oplam (W + \vartheta) \cap F_t}{\lambda(F_t)} \leq
\frac{|W|}{\vol(\cL)}$$ holds for all $\vartheta \in H$.

(c) If $W$ is open, the inequality
$$\liminf_{t\to \infty} \frac{ \sharp \Oplam (W + \vartheta) \cap F_t}{\lambda(F_t)} \geq
\frac{|W|}{\vol (\cL)}$$ holds for all $\vartheta \in H$.
\end{theorem}

\begin{remark} For recent results of the type presented in the
theorem  see also  \cite{HR}.
\end{remark}

\begin{proof} Part (a) of the Theorem is shown  in \cite{Moody}.
Inspecting the proof there one can easily infer  part (b) and (c) as
well. For the convenience of the reader we sketch a proof. This
proof can be seen as a  variant of the considerations in
\cite{Moody}: Consider  $\T= (\R^N\times H) / \cL$ and let  $\sigma
: \R^N \to [0,\infty)$ be a continuous function  with compact
support and $\int_{\R^N}  \sigma d s = 1$.
 Define  the function
$$ f: \T\longrightarrow [0,\infty) \quad , \ f(\xi) \ := \ \sum_{(s,h) \in -\xi}
   \sigma (s) 1_W (h),$$ where $1_W$ denotes the characteristic function of $W$.
   Then, $f$ is a measurable bounded function. (Note that the sum has only
   finitely many non-vanishing terms as both $\sigma$ and $1_W$ vanish outside
   compact sets.)

Define for $\xi = [s,h]_\cL$ the set $\Oplam (\xi):= \Oplam ( W+ h) - s$ and
note that this is indeed well defined. Then a short computation (compare
\cite{Moody}) shows that
$$\left|\frac{1}{\lambda(F_t)} \int_{F_t} f(\omega_s (\xi)) ds
 - \frac{\sharp \Oplam (\xi)\cap F_t}{\lambda(F_t)}\right|
\ \to \  0 \quad  \text{ as } t\to \infty
$$
for all $\xi \in \T$. Thus, the desired statements (a),(b), (c) will
follow from the  corresponding statements for the averages
$$a_t (\xi)\ := \ \frac{1}{\lambda(F_t)} \int_{F_t} f(\omega_s (\xi)) ds.$$
These statements in turn hold as $(\T,\R^N)$ is uniquely ergodic:

\smallskip

(a) Birkhoff's ergodic theorem directly implies convergence of the
averages $a_t (\xi)$ for almost every $\xi\in \T$. As convergence
for $\xi$ clearly implies convergence for all $\omega_s (\xi)$,
$s\in \R^N$, the almost sure convergence in $\xi \in \T$ implies
almost-sure convergence in $\vartheta \in H$.

\smallskip

(b) If we replace $1_W$ by a continuous function with compact
support, then $f$ is continuous and  we even have uniform
convergence in $\xi \in\T$ by Oxtoby's theorem. Approximating $1_W$
from above by continuous functions with compact support we obtain
the statement (b) uniformly in $\xi\in\T$ and hence also in
$\vartheta\in H$.

\smallskip

(c) This follows   by replacing the  approximation from above in (b)
by approximation from below. More specifically, by regularity of the
Haar measure on $H$ we can chose a compact set $K \subset W$ whose
measure is as close to the measure of $W$ as we wish. Now, invoking
Urysohn's Lemma we can chose a continuous $f$ with compact support
and $1_K\leq f \leq 1_W$.
\end{proof}

\Paragraph{ Topological Entropy.}
  In this section we introduce the background from entropy theory.  Given an
  $\R^N$-action $\phi$ on a compact metric space $X$ (whose metric we denote by
  $d$), we say $x,x'\in X$ are \emph{$(\varepsilon, t)$-separated} if $$
  \max_{s\in F_t} d(\phi_s(x), \phi_s(x')) \geq \varepsilon. $$ A subset $S \ssq
  X$ is called \emph{$(\varepsilon,t)$-separated} if its elements are all
  pairwise $(\varepsilon,t)$-separated. By $N(\varphi, \varepsilon, t)$ we
  denote the maximal cardinality of an $(\varepsilon, t)$-separated set. The
  \emph{topological entropy} of $\phi$ is defined as
   $$ \htop(\phi) \ := \ \lim_{\varepsilon \rightarrow 0}
   h_\varepsilon(\phi) = \sup_{\varepsilon >0} h_\varepsilon
   (\phi),$$
   where
   $$h_\varepsilon(\phi) \ = \  \limsup_{t \rightarrow \infty} \frac{1}{\lambda(F_t)}
   \log N(\phi, \varepsilon, t). $$

\medskip

We will be particularly interested in the topological entropy of a
dynamical system $(\Omega(\Oplam(W)),\R^N)$ arising from a CPS and a
proper $W$. In this case, there is a  torus parametrization
\begin{equation}
\label{eq:situationabove} \beta : \Omega (\Oplam (W))\longrightarrow
\T,
\end{equation}
 due to Proposition \ref{prop:DDSFlowMorphism} (applied with $\Lambda = \Oplam (W)$). As $(\T,\R^N)$ is
an isometric flow it has entropy zero.  This has some consequences
for the topological entropy of $(\Omega(\Oplam(W)),\R^N)$.  As it is
instructive to our considerations below  we discuss next some
abstract background in the subsequent two remarks.

\begin{remark}[Positive entropy comes from  fibres] \label{rem:Bowen}   If
$(Y,\psi)$ is a factor of $(X,\phi)$ we can relate the entropy of
the two systems. Indeed, we have
$$\htop(\psi)  \leq    \htop(\phi)$$ (e.g.\
\cite{katok/hasselblatt:1997}). It is then also  possible to obtain
an upper bound on $\htop(\phi)$ by considering the ``the topological
entropy realised in single fibres''. In order to be more specific,
let $\eta: X\longrightarrow Y$ be the factor map and denote
 for any $\xi \in Y$, the maximal cardinality of an
   $(\varepsilon,t)$-separated subset of the fibre $\eta^{-1}(\xi)$ by
   $N^\xi(\phi, \varepsilon, t)$. Now let
 $$ \htop^\xi(\phi) \ :=\  \lim_{\varepsilon \rightarrow 0} h^\xi_\varepsilon(\phi)\ , \quad \mbox{ where }
h^\xi_\varepsilon(\phi) \ := \  \limsup_{t \rightarrow \infty}
\frac{1}{\lambda (F_t)}
   \log N^\xi(\phi, \varepsilon, t).$$
Then, clearly
$$ \htop^\xi (\phi)\  \leq \  \htop (\phi)$$
for any $\xi \in Y$.  As shown in
\cite{bowen:EntropyForGroupEndomorphismsAndHomogeneousSpaces} we
furthermore  have the   bound
$$ \htop(\phi) \ \leq \  \htop(\psi) + \sup_{\xi \in Y}
    \htop^\xi(\phi).$$
If $\htop(\psi)=0$ the two preceding inequalities give
$$ \htop(\phi) \ = \ \sup_{\xi \in Y} \htop^\xi (\phi).$$
So, in this case the (positive) topological entropy of $\phi$ must
be realised already in  single fibres. Now, this is exactly the
situation described in \eqref{eq:situationabove}. In line with the
preceding considerations our approach to positive entropy of
$(\Omega(\Oplam(W)),\varphi)$ below will be based on showing
positive entropy already in the fibres. We will do so by exhibiting
what we call embedded fullshifts (see below for details).
\end{remark}

\begin{remark}[Positive entropy implies thick boundary] We also note that whenever $(X,\phi)$ is uniquely ergodic
and is measure theoretically isomorphic to a  factor $(Y,\psi)$ of
zero topological entropy then the topological entropy of $(X,\phi)$
must vanish as well. The reason is that the metric entropy (which we
will not define as we do not need it below) is invariant under
measure theoretic isomorphisms. Hence, the metric entropy of
$(X,\phi)$ and $(Y,\psi)$ must agree. As $(X,\phi)$ is uniquely
ergodic its topological entropy agrees with its metric entropy due
to a variational principle, see e.g. \cite{TZ}. In our situation
described in  \eqref{eq:situationabove}  we obtain then from
Corollary \ref{cor:metricisom} that the topological entropy of
$(\Omega (W),\R^N)$ must vanish whenever the boundary of $W$ has
measure zero. Now, this implies that examples of CPS with positive
topological entropy will necessarily have thick boundary and indeed
this will feature prominently in our constructions below.
\end{remark}

 \section{Embedded fullshifts and Topological Independence}
  \label{section:EmbeddedSubshiftsandTopologicalIndependence}
  In this section, we define a simple criterion, namely the existence of
  `embedded fullshifts', for positive entropy of the dynamical hull of a
  uniformly discrete set in $\R^N$.  For hulls coming from (weak) model sets, we
  then relate this to the local structure of the window and introduce the
  concepts of local topological and metric independence. These will be the main
  tools to prove positivity of entropy in the constructions in the later
  sections.

\medskip

Whenever we meet a CPS $(\R^N,H, \mathcal{L})$ in this and the
remaining sections the group $H$ will be abelian,  metrizable and
$\sigma$-compact, compare the discussion on Page
\pageref{page:cpsintroduced}.

\Paragraph{ Embedded fullshifts.} Embedded fullshifts are our key
concept in providing positive topological entropy.

\smallskip

\begin{mydef}[Embedded fullshift]\label{def:DSSubshift}
Let $\Lambda$ be a uniformly discrete subset of $\R^N$. An
\textit{embedded fullshift} in $\Omega (\Lambda)$ is a  pair $(\Xi,
S)$ consisting of a closed subset  $\Xi$ of $\Omega (\Lambda)$ and a
subset  $S$ of $\R^N$ such that the following holds:

\begin{itemize}
\item The set $S$ has positive asymptotic density, i.e.
    $$\nu_S =\limsup_{t\to \infty} \frac{\sharp S\cap F_t}{\lambda(F_t)} >0.$$

\item The set
$$U:=\bigcup_{\Gamma\in\Xi} \Gamma \subset \R^N$$
is uniformly discrete.

\item For any subset $S'$ of $S$ there exists a $\Gamma\in\Xi$ with
$$\Gamma \cap S = S'.$$

%\item The set $\{\Gamma \in\Xi : \Gamma\cap S\neq \emptyset\}$ is
%dense in $\Xi$.

\end{itemize}
The elements of $S$  above are called {\em free points} of the
embedded fullshift. The set $U$ is called  {\em grid} of the
embedded fullshift. The quantity  $\nu_S$ is the \textit{asymptotic
density} of the embedded fullshift.

If $(\Xi,S)$ is an embedded fullshift in $\Omega(\Lambda)$ with $\Xi
\ssq \Omega'$ for some $\Omega'\ssq \Omega (\Lambda)$ we say that
$\Omega'$ \textit{contains an embedded fullshift}.
\end{mydef}

\begin{remark} \label{rem:es} Consider an embedded fullshift with free points $S$
  and grid $U$.

(a) We clearly have $S\ssq U$ (and therefore $S$ is uniformly
discrete). The points of $S$ are free in the sense that we can
choose any subset of $S$ and exactly this will be the  subset from
$S$ appearing  in some $\Gamma \in \Xi$. In later arguments we will
not only have to control occurrence of points of $S$ but also
non-occurrence of points of $S$.  We will need the set $U$  in order
to treat this non-occurrence.

(b)  We  call an embedded fullshift $(\Xi,S)$ with grid $U$
\textit{maximal} if $(\Xi, S\cup\{u\})$ is not an embedded fullshift
for any $u\in U\setminus S$. In this case, we may think of the
elements of $U\setminus S$ as points \textit{forced by the embedded
fullshift}. It is not hard to see (by an induction procedure) that
any embedded fullshift can be extended to a maximal one.

(c) Let $(\Xi, S)$ be an embedded fullshift. Then, $(\tilde \Xi, S)$
with
$$\tilde \Xi :=\closure( \{ \Gamma\in \Xi : \Gamma \cap S \neq
\emptyset\}),$$ where the closure is taken in the hull of $\Lambda$,
will also be an embedded subhift (with $\tilde \Xi \ssq \Xi$).
Indeed, the only possible difference between $\tilde \Xi $ and $\Xi$
are those elements of $\Xi$, which do not contain any element of
$S$.

(d)  Consider a  CPS $(\R^N, H, \cL)$ and a proper window $W$ and
$\Lambda = \Oplam (W)$. Then, for $\xi = [s,h]_\cL$,  all elements
of $\beta^{-1} (\xi)$ are contained in the uniformly discrete set
$\Oplam (W+h) -s$ by Proposition \ref{prop:DDSFlowMorphism}.   So,
for any subset $\Xi$ of $\beta^{-1} (\xi)$ we  have uniform
discreteness of $\bigcup_{\Gamma\in\Xi} \Gamma$. So, the uniform
discreteness of the grid is automatically satisfied for a fullshift
embedded in such  a fibre. Also, in this situation  if $(\Xi,S)$ is
an embedded fullshift in the fibre $\beta^{-1} (\xi)$, then
$(\beta^{-1} (\xi), S)$ is an embedded fullshift as well.  From
Proposition \ref{prop:DDSFlowMorphism}  and Lemma
\ref{lem:dichotomy} we then infer that the grid for this fullhift is
given by  $\Oplam (W + h) - s$.

(e)  Whenever the pair  $(\Xi,S)$ is an embedded fullshift, then so
is the translated pair $(\varphi_s (\Xi), \varphi_s (S))$ for any
$s\in\R^N$.

(f) We will be mostly interested in embedded fullshifts contained in
either $\Omega (\Lambda)$ or in the fibres $\eta^{-1}(\xi) \ssq
\Omega (\Lambda)$ of some flow morphism $\eta :
\Omega(\Lambda)\longrightarrow Y$.
\end{remark}

\medskip

The following  provides a  simple characterization for existence of
an embedded fullshift.

  \begin{prop}
   \label{prop:DSSubshift} Let $\Lambda$ be a uniformly discrete subset of $
   \R^N$. Then, $\Omega (\Lambda)$ contains an embedded fullshift if
   and only if there exist   $S\ssq \R^N$ and a uniformly discrete $U \ssq \R^N$
   with the following two properties:
   \begin{enumerate}
    \item The set $S$ has positive asymptotic density.

    \item For all finite $F\ssq S$ and $a\in \{0,1\}^F$, there  exists a $\Gamma \in
    \Omega(\Lambda)$ with $\Gamma \ssq U$ and such that for
    $s\in F$
     $$ s \in \Gamma \ \Longleftrightarrow \  a_s = 1.$$
   \end{enumerate}
\end{prop}

\begin{proof} If $\Omega(\Lambda)$ contains an embedded fullshift
there clearly exist $S\ssq \R^N$ and a uniformly discrete $U\ssq
\R^N$ satisfying (1) and (2). Conversely, if there exist $S\ssq
\R^N$ and a uniformly discrete $U\ssq \R^N$ satisfying (1) and (2)
we may define
$$\Xi':=\{ \Gamma\in\Omega (\Lambda) : \Gamma \cap S\neq \emptyset
\mbox{ and } \Gamma\ssq U\}.$$ Now, let  $\Xi$  be the closure of
$\Xi'$. Then, all elements in  $\Xi$ are contained in $U$ and a
simple compactness argument shows that for any subset $S'$ of $S$
there exists a $\Gamma\in\Xi$ with $\Gamma \cap S = S'$.  Hence,
$(\Xi, S)$ is an embedded fullshift contained in $\Omega (\Lambda)$.
\end{proof}

\begin{remark}\label{rem:es-weakmodel}
 Let  $\Lambda$ be  a (weak) model set coming from a CPS $(\R^N,H,
\mathcal{L})$ such that $\Omega (\Lambda)$  satisfies the conditions
(1) and (2) of the preceding proposition. Let $(\Xi,S)$ be the
embedded fullshift constructed in the proof of the preceding
proposition i.e. $\Xi:= \closure (\Xi')$ with $\Xi': =  \{
\Gamma\in\Omega (\Lambda) : \Gamma \cap S\neq \emptyset \mbox{ and }
\Gamma\ssq U\}$, and let $U'=\bigcup_{\Gamma\in\Xi'} \Gamma$.
 Then, the inclusions
 $$S\ssq U' \ssq t + L $$ turn out to be valid for any $t\in S$. Indeed, shifting $S$ and $U$ by $-t$ for
 $t\in S$, we may assume without loss of generality $t=0$ and $0\in S$. By $0\in
 L$, we infer from Proposition \ref{prop:eineroderallepunkte} that any $\Gamma
 \in \Xi$ containing $0$ must be contained in $L$. Hence, we have that
 \[
    \wh U \ = \ \bigcup_{\Gamma\in\Xi: 0\in\Gamma} \Gamma \ \ssq \ L \
 \]
 and since $\wh U\ssq U$ is a discrete set, a simple compactness argument shows
 $\wh U= U'$. Hence $U'\ssq L$, and since further any $s\in S$ is clearly
 contained in $U'$, this shows the claimed statement.
\end{remark}

%\begin{mydef}
%   If the window $W$ is proper, $\beta$ is the flow morphism from
%   Proposition \ref{prop:DDSFlowMorphism} and it is possible to find all
%   $\Gamma$ above in $\beta^{-1}(\xi)$ for some $\xi\in\T$, then we say \textit{the
%   fibre $\beta^{-1}(\xi)$ contains an embedded fullshift}.
% \end{mydef}

\bigskip

The relevance of embedded fullshifts  comes from the following
lemma.

  \begin{lemma}[Embedded fullshift implies positive entropy]\label{lem:esimpliespositiveh}
   Let $\Lambda$ be a uniformly discrete subset of $\R^N$.  If $\Omega(\Lambda)$
   contains an embedded fullshift of asymptotic density $\nu_S$, then
    $$ \htop(\varphi) \geq \nu_S \cdot  \log 2.$$
  \end{lemma}
  \begin{proof}
Let  $S$ be the set of free points and   $U$ the grid of the
embedded fullshift.   Let $r>0$ such that different points of $U$
have   distance at least $r$.    Consider  $\Gamma, \Gamma'
   \in \Omega(\Lambda)$ with $\Gamma, \Gamma'\ssq U$ and $s\in
   \Gamma$ and $s\notin \Gamma'$ for some $s\in S$. By uniform
   discreteness of $U$ the set $\Gamma'$ then does not contain a
   point in the ball around $s$ with radius $r$. This gives
$$d(\varphi_{s}(\Gamma),   \varphi_{s}(\Gamma')) \geq r.$$
 Hence, any pair $\Gamma, \Gamma'
   \in \Omega(\Oplam(W))$ which satisfies the above for some $s \in S\cap  F_t$
   is $(r,t)$-separated.
Consider now  an arbitrary  $\nu' < \nu_S$. Then, there exist
arbitrarily large $t$   with
$$\sharp S\cap F_t \geq \nu' \cdot \lambda(F_t).$$
By the assumption on existence of an embedded fullshift we can chose
for  any finite subset $F$ of $S\cap F_t$ an element
$\Gamma_F\in\Omega (\Lambda)$ with $\Gamma_F \cap S = F$. Then, the
elements $\Gamma_F$ are $(t,r)$ separated by  the considerations at
the beginning of the proof. Hence, we have
$$N(\varphi,r,t)\geq 2^{\nu' \cdot \lambda(F_t)}.$$
This implies
$$h_r (\varphi) \geq \nu'  \cdot \log 2.$$
As $\nu' <\nu_S$ was arbitrary we infer $h_r (\varphi) \geq \nu_S
\log 2.$ Now, the desired statement follows from $\htop (\varphi)
\geq h_r (\varphi)$.
 \end{proof}

  \Paragraph{ Independence of sets and existence of embedded fullshifts.}
In this section we provide a condition for existence of an embedded
subshift.

\smallskip

Consider a CPS $(\R^N,H,\cL)$ and denote the neutral element of $H$ by $0$. We
sometimes write $0\in H$ in order to distinguish it from the origin in $\R^N$.

\smallskip

Consider a (weak) model set arising from the given CPS via a
relatively compact $W\ssq H$. Then the  problem of finding an
embedded fullshift with set of free points $S\ssq L$  in the
associated dynamical system is actually related to analyzing the
local structure of the window $W$ in some neighborhood of the points
$s^*$ for $s\in S$. In order to get a first idea on this issue the
following observation may be helpful:

Let $F\ssq L$ be a finite set, $a\in \{0,1\}^F$ arbitrary and $\vartheta\in H$
be given. Now, assume that
$$ \emptyset \neq  \left(\bigcap_{s\in F : a_s = 1} ( W - s^*) \setminus
\bigcup_{s\in F: a_s = 0} (W-s^*) \right) \cap  (L^* -\vartheta).$$
Then, there exists an $l\in L$ satisfying for any $s\in F$ that
 $(l^* -\vartheta) \in
W-s^* \Longleftrightarrow a_s = 1$. This gives for $s\in F$
$$s\in \Oplam (W + \vartheta) -l \Longleftrightarrow a_s = 1.$$
Thus, the set $\Gamma := \Oplam (W + \vartheta) - l$ respects the
choice of $F\ssq L$ given by $a$. Our dealings below will build
on this observation. However, two additional points will come up:

\begin{itemize}

\item We have to simultaneously deal with all finite subsets $F$ of a subset $S$
  of $L$.  In order to still provide the uniform discrete subset $U$ necessary
  for an embedded fullshift, we will need to require that the set $S^*=\{s^*\mid
  s\in S\}$ is relatively compact (see Lemmas
  \ref{lem:localIndependenceWithRespectTo1} and
  \ref{lem:localMetricIndependence}).

\item We will allow for one overall shift by $h\in H$.

\end{itemize}

\medskip

Motivated by the preceding considerations we give the following
definition. The finite  index set $F$ appearing in the definition
will later be a subset of $L$ or $L^*$.

\begin{mydef}[Independence with respect to $\DDD$]
Let $\DDD\ssq H$ be given. A finite family  $(A_s)_{s\in F}$ of
subsets of  $H$ is \emph{independent with
  respect to $\DDD$} if for all $a \in \{0,1\}^F$ we have
   $$\emptyset \neq \left( \bigcap_{s\in F : a_s = 1} A_i \setminus
     \bigcup_{s\in F : a_s = 0} A_i \right) \cap \DDD. $$ An infinite family of
   sets is called \emph{independent with respect to $\DDD$} if the condition
   above holds for each finite subfamily.  We say the window $W$ is
   \emph{independent in $P \ssq L^*$ with respect to $\DDD$}, if the family $W -
   p, p\in P,$ is independent  with respect to $\DDD$.
\end{mydef}

\smallskip

The following lemma relates these concepts to the existence of
embedded fullshifts. The lemma is our main tool to construct
embedded fullshifts (and hence, by Lemma
\ref{lem:esimpliespositiveh} examples with positive topological
entropy). In fact, we will apply the lemma in two situations, namely
for proper $W$ and $W$ with empty interior but of positive measure.
These situations will then be studied in the two subsequent
sections.  The lemma is formulated in a general version that
includes two parameters $\vartheta$ and $h$. However, for a first
reading it might be helpful to set them both to zero.

\begin{lemma}[Basic criterion for embedded fullshifts]
   \label{lem:localIndependenceWithRespectTo1}
   Let $(\R^N,H,\cL)$ be a CPS, $W\ssq H$ relatively compact and $\vartheta,h\in
   H$. If $\Oplam(W+h)$ possesses a subset $S$ of positive asymptotic density
   such that $S^*=\{s^*:s\in S\}$ is relatively compact and $W+h$ is independent
   in $S^*$ with respect to $L^*+(h-\vartheta)$, then
   $\Omega(\Oplam(W+\vartheta))$ contains an embedded fullshift.
  \end{lemma}
\begin{remark} Note that we do not require that $W$ has non-empty interior.
So, the sets $\Oplam (W +v)$ may not be relatively dense. However,
they are uniformly discrete and this is all we need to consider the
hull.
\end{remark}

 \begin{proof}
 We will show that the  conditions (1)  and (2) of Proposition \ref{prop:DSSubshift}
  for existence of  an
 embedded fullshift are met for $S$ as in the statement of the lemma
 and
 $$U:= \Oplam (W +h -S^*).$$
 Note that $U$ is indeed uniformly discrete as $W + h -S^*$ is relatively
 compact, compare Lemma \ref{lem:CPSMeyer}.

\smallskip

Condition (1) is met by assumption.  To show condition (2) fix a
finite subfamily $F\ssq S$ and $a \in \{0,1\}^F$. Then, by
independence of $W+ h$ in $S^*$ with respect to $L^*+(h-\vartheta)$,
we have
    \[
    \emptyset \neq \left(\bigcap_{s \in F : a_s = 1} (W + h - s^*) \setminus
        \bigcup_{s\in F : a_s = 0} (W + h - s^*) \right) \cap (L^*+ h -\vartheta)
     \ .
    \]
    Thus there exists an
    \begin{equation}\label{e.limitzero}
\bar m^*\in (L^*+ h -\vartheta)
    \end{equation}
such that
\begin{equation} \label{e.loc_ind_1}
     \bar m^*\in W+ h -s^* \ \equi \
    a_s=1 \
  \end{equation}
  for all $ s\in F$. Further, by the symmetry $L^*=-L^*$ we have
\begin{equation}\label{e:limit}
\bar m^*=h -\vartheta-m^*
\end{equation}
   for some $m\in L$. Combining this with \eqref{e.loc_ind_1}
   we obtain
%Then $\bar m_j^*\in
%  W+t-l_{n_k}^*$ is equivalent to $l^*_{n_k}\in W+t-\bar
%  m_j^*=W+\vartheta+m_j^*$, and further to $l_{n_k}\in
%  \Oplam(W+\vartheta+m_j^*)=\Oplam(W+\vartheta)-m_j$.
 % We therefore obtain
 \begin{equation}\label{eq:contain}
      s \in \Oplam(W+\vartheta) + m \quad \textrm{ if and only if } \quad a_s=1 \ .
  \end{equation}
Moreover, we have
  $$\Gamma:= \Oplam(W+\vartheta) + m = \Oplam (W + \vartheta +  m^*) \stackrel{\eqref{e:limit}}{=}
  \Oplam (W + h - \bar m^*) \ssq U,$$ where we used that $\bar m^*$ belongs to
  $W+h-S^*$ by (\ref{e.loc_ind_1}) to obtain the last inclusion.  Thus
  $\Gamma$ belongs to $\Omega(\Oplam(W+\vartheta))$ with $\Gamma\ssq U$, and
  due to \eqref{eq:contain} we have
$$s\in
\Gamma \quad \textrm{ if and only if } \quad a_s=1 \ .$$
%Now, $(\Gamma_j)$ is a sequence in the compact  space
%  $\Omega(\Oplam(W+\vartheta))$. Hence, it possesses an  accumulation
%  point $\Gamma \in \Omega(\Oplam(W+\vartheta))$. As
%  $\Gamma_j$ is a subset of $U$  (for large $j$) and $U$ is uniformly
 % discrete, we then obtain from \eqref{eq:contain}
 % $$
This finishes the proof.
  \end{proof}

  A slightly more specific notion of independence is given in the following definition.
  It will be needed in particular to obtain further information
about  embedded fullshifts in the case of proper model sets.

\begin{mydef}[Local independence in $0\in H$] Let $\DDD\subset
H$ be given.   An infinite family $(A_s)_{s\in P}$ of subsets of $H$
is said to be
  \textit{locally independent in $0\in H$} with respect to $\DDD$ if
$$ 0 \ \in \ \closure\left(\left(\bigcap_{s\in F: a_s=1} A_s\smin \bigcup_{s\in
    F:a_s=0} A_s\right) \cap D\right)
$$
for any finite subset $F\ssq P$ and any $a\in\{0,1\}^F$.  The window $W$ is
said to be \textit{locally independent in $P\ssq L$ with respect to $\DDD$} if
the the family $W-p$, $p\in P$, is locally independent in $0\in H$ with respect
to $\DDD$. Hence, the window $W$ is {locally independent in $P$ with respect to
  $\DDD$} if and only if
\begin{equation}\label{e.local_independence}
0\ \in\ \closure\left( \left( \bigcap_{s\in F : a_s = 1} (W-s^*) \setminus
   \bigcup_{s\in F         : a_s = 0} (W-s^*) \right) \cap \DDD \right)
\end{equation}
for any finite family $F\ssq P$ and any $a\in\{0,1\}^F$.
\end{mydef}

\begin{corollary}
   \label{cor:localIndependenceWithRespectTo1}
Let $(\R^N,H,\cL)$ be a  CPS, $W\ssq H$ relatively  compact, proper
and $\vartheta,h\in    H$.  Assume that $\Oplam(W+h)$ possesses a
subset $S$ of positive asymptotic density such that
    $W+h$ is locally independent  in
     $S^*=\{s^* : s\in S\}$ with respect to
$L^*+(h-\vartheta)$.  Let  $\beta_\vartheta$ denote the flow
morphism    described in \eqref{r.flow_morphism}. Then there is an
embedded fullshift contained in $\beta_\vartheta^{-1}([0, h
-\vartheta]_\cL)$.
  \end{corollary}
  \begin{proof}
  This follows by extending the proof of the previous lemma. Here are the
  details: Fix a finite subfamily $F\ssq S$ of and $a \in \{0,1\}^S$. Then, due
  to (\ref{e.local_independence}) we can choose $\bar m^*$ such that it
  satisfies~(\ref{e.limitzero}) and (\ref{e.loc_ind_1}) and additionally require
  that $\bar m^*$ is arbitrarily close to $0$. This means that we can find a
  sequence of $\bar m_j^*$ such that (\ref{e.limitzero}) and (\ref{e.loc_ind_1})
  hold for all $j\in\N$ and at the same time
    \begin{equation}
\label{e.convergence} \jLim \bar m^*_j \ =\ 0 \ .
  \end{equation}
  Further, we can fix a relatively compact neighbourhood $V$ of $0$ and assume
  without loss of generality that $\bar m^*_j\in V$ for all $j\in\N$.

    If now $m_j\in L$ are chosen such that $\bar m_j^*=h-\vartheta-m_j*$,
    analogous to (\ref{e:limit}), then we obtain
    \begin{equation} \label{e:equivv}
      s \in  \Gamma_j  :=  \Oplam(W+\vartheta)+m_pj \quad \textrm{if and only if}  \quad a_j= 1 \ .
    \end{equation}
Moreover, we have
\[
\Gamma_j\ = \ \Oplam(W+\vartheta+m_j^*) \ = \ \Oplam(W+h-\bar m_j^*) \ \ssq \ \Oplam(W+h-V) \ =: \ U_1 \ ,
\]
where $U_1$ is discrete since $W+h-V$ is relatively compact (compare
Lemma~\ref{lem:CPSMeyer}).

Now, $(\Gamma_j)$ is a sequence in the compact space
$\Omega(\Oplam(W+\vartheta))$. Hence, it possesses an accumulation point $\Gamma
\in \Omega(\Oplam(W+\vartheta))$. As $\Gamma_j$ is a subset of $U_1$ and $U_1$
is uniformly discrete, convergence of the $\Gamma_j \to \Gamma$ and
\eqref{e:equivv} yield
$$s\in
\Gamma \quad \textrm{ if and only if } \quad a_s=1.$$

As $W$ is proper,  $\beta_\vartheta$ is continuous. This gives
  \begin{eqnarray*}
    \beta_\vartheta(\Gamma) & = & \jLim \beta_\vartheta(\Oplam(W+\vartheta) + m_j)\\
    &=& \jLim \beta_\vartheta(\Oplam(W+\vartheta + m_j^*)\\
    & = &  \jLim [0,m_j^*]_\cL \\
(\mbox{by } \eqref{e:limit})\;\:     &  = & \ \jLim [0,h-\vartheta-\bar m_j^*]_\cL \\
(\mbox{by } \eqref{e.convergence}) \;\:  &=&   [0,h-\vartheta]_\cL.
  \end{eqnarray*}
This shows that the $\Gamma$ constructed above are all contained in
the fibre
  $\beta_\vartheta^{-1}([0,h-\vartheta]_\cL)$. Thus, we obtain an
  embedded fullshift in that fibre.
\end{proof}

  \Paragraph{ Local topological independence and proper $W$.} \label{LocalTopInd}
  In this section we consider the case that $W$ is proper. We provide a
  sufficient condition for applicability of Corollary
  \ref{cor:localIndependenceWithRespectTo1}. This condition is given in Lemma
  \ref{lem:localTopologicalIndependence}. Our application to the random model in
  Theorem~\ref{t.random_windows} will be based on that lemma.

\smallskip

  We say a finite family of sets $A_s$, $s\in F$, of subsets of $H$ is
  \emph{locally topologically independent in $0\in H$} if for all $a \in \{0,1 \}^F$
  we have
   $$ 0 \in \closure \left( \interior \left( \bigcap_{s\in F : a_s = 1}
   A_s
       \setminus \bigcup_{s\in F : a_s = 0} A_s \right) \right) .$$ An infinite
   family of sets is called \emph{locally topologically independent in $0$} if
   the condition above holds for each finite subfamily. A window $W$ is
   \emph{locally topologically independent in} $P\ssq L^*$, if the family
   $W-p, p\in P,$ is locally topologically independent in $0$.

  \begin{lemma}
   %\label{lem:localTopologicalIndependence}
   Any family  of subsets of $H$, which is  locally topologically independent in $0\in H$,
   is locally independent in $0$ with respect to any dense  $\DDD\ssq H$.
  \end{lemma}
  \begin{proof}
    Consider an arbitrary finite subfamily  $A_s$, $s\in F$, of the original family
    and let $a\in \{0,1\}^F$ be given. Define
    $$A(a) = \bigcap_{s\in F : a_s=1} A_s \setminus \bigcup_{s: a_s =0} A_s.$$
    By assumption we have $0\in \closure(\interior(A(a)))$. Since $\DDD$ is
    dense in $H$, the intersection $\interior(A(a)) \cap \DDD$ is dense in
    $\interior(A(a))$. Thus, we can choose a sequence $\jfolge{h_j}$ in
    $\interior(A(a)) \cap \DDD $ such that $\lim_{j \rightarrow \infty} h_j =
    0$.
  \end{proof}

  \begin{lemma}[Topological criterion for embedded fullshifts]
     \label{lem:localTopologicalIndependence}
   Let $(\R^N,H,\cL)$ be a  CPS, $W\ssq H$ a proper window and $h\in
   H$. Assume that there exists a subset   $S$ of $\Oplam(W+h)$ of positive asymptotic
   density such that $W+h$  is locally
   topologically independent in $S^*= \{s^*  : s\in S\}$.
Then, the fibre $\beta_\vartheta^{-1}([0,h-\vartheta]_\cL)$ contains
   an embedded fullshift    for every $\vartheta \in \R$.
  \end{lemma}
\begin{proof}
As $W+h$ is locally topologically independent in $S^*$ and $L^* + (h
-\vartheta)$ is dense in $\R^N$ for all $\vartheta \in H$, the
preceding lemma gives that  $W+h$ is locally    independent in $S^*$
with respect to $L^*+(h-\vartheta)$ for  all $\vartheta\in \R$. As
$W$ is proper, we can now apply Corollary
\ref{cor:localIndependenceWithRespectTo1} to obtain that the fibre
$\beta_\vartheta^{-1}([0,h-\vartheta]_\cL)$ contains   an embedded
fullshift.
% By Lemma \ref{lem:esimpliespositiveh},   $(\Omega(\Oplam(W+\vartheta),\varphi)$ has then positive entropy.
\end{proof}

\Paragraph{ Metric independence and general $W$.}
  The aim of this section is to adapt the above concepts for the case of weak
  model sets, that is, to compact windows with empty interior. In this case, we
  need to replace open sets by sets of positive measure and invoke
  uniform distribution  in order to prove analogous statements. As a
  result we will obtain a criterion for embedded fullshifts for
  general relatively  compact $W$ with positive measure. This criterion is given in Lemma
  \ref{lem:localMetricIndependence}.

\medskip

%  \begin{theorem}[Birkhoff's Ergodic Theorem] %%\TODO: This could be written
                                              %%into the basics-chapter
%   \label{the:birkhoffsErgodicTheorem}
%   Let $(X,\cB)$ be a measurable space, $T:X\to X$ a measureable transformation
%   and $\mu$ a probability measure on $(X,\cB)$ which is invariant and ergodic
%   with respect to $T$. Then for every $\mu$-integrable $f:X\to\R$ and
%   $\mu$-almost every $x \in X$, we have
%   \[
%      \lim_{n \rightarrow \infty} \frac{1}{n} \sum_{j=1}^n f(T^j(x)) = \int_X f \ \mathrm{d}\mu \ .
%   \]
%  \end{theorem}
%  Note that in particular this means that if $A\ssq X$ is a measurable set, then
%  \[
%     \nLim 1/n\sharp\{0\leq j < n\mid T^j(x) \in A\} \ = \ \mu(A)
%  \]
%  for $\mu$-almost every $x\in X$ (choose $f$ to be the indicator function of
%  $A$).

%  \begin{mydef}[Essential closure]
%    Let $M \ssq H$. The \emph{essential closure of $M$} is
%    defined as
%    $$ \closess(M) := \{x \in H   \mid     | M \cap U| > 0 \text{ for all open neighborhoods of
%    $U$ of $x$}  \}. $$
%  \end{mydef}

\begin{mydef}
 A finite family $(A_s)_{s\in F}$ of subsets of $H$ is
\emph{ metrically independent} if for all $a \in \{0,1\}^F$ we have
   $$ 0< \left| \left(\bigcap_{s\in F : a_s = 1} A_s \setminus \bigcup_{s\in F :
  a_s = 0} A_s \right) \right|. $$ An infinite family of subsets of $H$ is
called \emph{metrically independent} if the condition above holds for each
finite subfamily. Further, we say the window $W$ is \emph{metrically independent
  in} $P\ssq L^*$, if the family $W -p, p \in P$, is metrically independent.
   \end{mydef}

   \begin{lemma}[Metric criterion for embedded fullshifts]
     \label{lem:localMetricIndependence} Let $(\R^N,H,\cL)$ be a CPS, $W\ssq H$
     a relatively compact window and $h\in H$. Assume that there exists a subset
     $S$ of $\Oplam(W+h)$ of positive asymptotic density such that
     $S^*=\{s^*:s\in S\}$ is relatively compact and $W+h$ is metrically
     independent in $S^*$.
    %Then,  $W+h$ is locally independent around $V$  in
    %$S^*$ with respect to $L^*+(h-\vartheta)$ for almost
    %every $\vartheta\in H$. If furthermore  $S$ has  positive
    %asymptotic density,
    Then  $\Omega(\Oplam(W+\vartheta))$
    contains an embedded fullshift for almost every $\vartheta\in H$.
\end{lemma}
\begin{proof}
   %Let $\pi : \mathbb{R} \rightarrow \mathbb{S}^1$ be the canonical projection
   %on the circle.
%    As $A=\twomatrix{a}{b}{c}{d}$, we have $L^*=\{nc+md\mid
%   (n,m)\in\Z^2\}$. Now, replacing $\twovector{a}{b},\twovector{c}{d}$ by some
%   other pair of generators of $\cL$ (which does not affect the statement of the
%   lemma), we may assume that $0< c < d$ and $0<b<a$. Further, by rescaling we
%   may assume $d=1$. Note that due to the irrationality of the lattice $\cL$, we
%   have $c\notin \Q$ in this case. Then, given $p\geq 2$, we have
%   \begin{equation}\begin{split}\label{e.rotation_representation1}
%   \lefteqn{(L^*+t-\vartheta)\cap B_{1/p}(0) \ \ssq
%     \ (L^*+t-\vartheta)\cap[-1/2,1/2) \ =} \\ & = \  \{nc - m+t-\vartheta \mid
%     n,m\in\Z,\ nc-m+t-\vartheta\in[-1/2,1/2)\}\\ & = \ \{nc-m +t -\vartheta +1/2 \mid
%       n,m\in\Z,\ nc-m-\vartheta\in[0,1)\} -1/2\\ & = \ \{nc+t-\vartheta+1/2 \bmod 1\mid
%         n\in\Z\} -1/2\ .
%     \end{split}
%   \end{equation}
%   If we identify $[-1/2,1/2)$ with $\kreis$ by $p:[-1/2,1/2)\to\kreis,\ x\mapsto
%       x+1/2\bmod 1$, then the set in the last line corresponds to the orbit the
%       irrational rotation $R_c:\kreis\to\kreis,\ x\mapsto x+c\bmod 1$ with
%      starting point $t-\vartheta+1/2\bmod 1$. Moreover, given $l^*(n)=nc-m(n)\in L^*$
%      with $m(n)$ defined by $l^*(n)+t-\vartheta+1/2=nc-m(n)+t-\vartheta+1/2\in[0,1)$,
%       we have $l=na-m(n)b$, which is increasing in $n$ with increments either
%       $a$ or $a-b$.
Let  $F$ be a finite subset of $S$ and let $a\in \{0,1\}^F$ be
given. Consider the family  $W+h-s^*$, $s\in F$, and define
    $$\mathcal{W}(a) \ = \ \bigcap_{s\in F : a_s = 1} (W+h-s^*) \setminus
   \bigcup_{s\in F: a_s = 0} (W+h-s^*). $$
   Since $W + h $ is metrically
   independent in $S^*$, we have
    $$ 0 < |\mathcal{W}(a) |. $$
    By uniform distribution, Theorem \ref{t:uniform-distribution}, we thus
    obtain that the density of
$$\Oplam (\mathcal{W}(a)  -h   +  \vartheta)$$
is positive for almost every $\vartheta \in H$.
%    and thus, due to Birkhoff's Ergodic
%   Theorem,
%   \begin{equation} \begin{split}\label{e.rotation_representation2}
%    \lefteqn{\lim_{n \rightarrow \infty} 1/n\cdot\sharp\{0\leq j< n\mid
%      l^*(j)+t-\vartheta\in \cW(a)\cap B_{1/p}(0)\}} \\ & = \lim_{n \rightarrow
%      \infty} 1/n\cdot\sharp\{0\leq j < n\mid R^j_c(t-\vartheta+1/2)\in
%    p(\cW(a)) \} \\ & = \ \Oplam(\cW(a)) \ > \ 0
%     \end{split}
%   \end{equation}
%   for all
 %  $p \in \mathbb{N}$ and $\lambda$-almost every $\vartheta \in
 %  \mathbb{R}$.
By  excluding a set of  measure zero, we therefore
   obtain a set $\Theta (a)  \ssq H$ of full measure such that for every $\vartheta\in \Theta (a)$
   the set $L^*+ h -\vartheta$ intersects $V\cap \cW(a)$.
% Thus, there exists a    sequence  $(l_j)$  in $\R^N$ with
%   $$l_j^* \in (L^*+h -\vartheta) \cap \mathcal{W}(a)$$
%and $$l_j^* \rightarrow 0, \; j \rightarrow \infty.$$
Intersecting over  the countable family of all finite $F\ssq S$ and
$a\in \{0,1\}^F$ we obtain a set $\Theta \ssq H$ of full measure
such that for each $\vartheta\in\Theta  $ the set  $L^*+ h
-\vartheta$ intersects $V\cap \cW(a)$ for arbitrary $F\ssq  S$ and
$a\in \{0,1\}^F$. Hence,  $W+h$ is locally independent around $V$ in
$S^*$ with respect to $L^*+h-\vartheta$ for each $\vartheta
\in\Theta$. Given this,
Lemma~\ref{lem:localIndependenceWithRespectTo1} implies the
assertion.
  \end{proof}

 \section{Embedded fullshifts and unique ergodicity}
  \label{section:FurtherPropertiesOfDDSWithPositiveTopologicalEntropy}
In this section we study how existence of a embedded fullshifts of
sufficiently high density  prevents unique ergodicity.

\smallskip

Recall that we have defined the asymptotic density of a subset
$\Gamma$ of $\R^N$ by
$$\nu_\Gamma :=\limsup_{t\to\infty} \frac{ \sharp \Gamma\cap
F_t}{\lambda (F_t)}.$$

Let now $(\R^N,H,\cL)$ be a cut and project scheme and $W\ssq H$ be relatively
compact and $(\Omega(\Oplam (W)),\R^N)$ the associated dynamical system. If this
system is uniquely ergodic, then, by
Lemma~\ref{l.basic_dynamical_properties}(a), the density $\lim_{t\to\infty}
\frac{ \sharp \Gamma\cap F_t}{\lambda(F_t)}$ exists for every $\Gamma \in \Omega
(\Oplam (W))$ and is independent of $\Gamma$ (as this density is just the patch
frequency of the patch $(\{0\},r/2)$, where $r$ is the minimal distance between
points in $\Gamma$).
% Then the {\em maximal density of (elements of)} $\Omega(\Oplam(W))$ is
% defined as
%  \[
%  \nu_{\Omega(\Oplam(W))} \ = \
%  \sup_{\Gamma\in\Omega(\Oplam(W))}\limsup_{S\to\infty} 1/S\cdot
%  \sharp(\Gamma\cap [0,S]) \ .
%  \]
Based on this observation we can now show that
$(\Omega(\Oplam(W)),\R^N)$ can not be uniquely ergodic if it
contains an embedded fullshift with set of free points
  $S$ and grid $U$ such that $\nu_S$ is large compared to $\nu_U$.

  \begin{prop}
   \label{prop:UniqueErgodicityDependingOnRelativeDensity}
   Let $(\R^N,H,\cL)$ be a CPS and $W$ a relatively compact window and suppose
   that $\Omega(\Oplam(W))$ contains an embedded fullshift with set of free
   points $S$ and grid $U$ and $\nu_S > \nu_U / 2$.  Then
   $(\Omega(\Oplam(W)),\varphi)$ is not uniquely ergodic.  This applies in
   particular if $W$ is proper and $\Omega(\Oplam (W))$ contains a fullshift
   embedded in a fibre with asymptotic density $\nu_S > |W|/2\vol(\cL)$.
  \end{prop}
  \begin{proof} Let $(\Xi,S)$ be the embedded fullshift in question.
Let $\Gamma_1, \Gamma_0 \in \Xi$ be given with $\Gamma_1 \cap S = S$
and $\Gamma_0 \cap S = \emptyset$. Then
   \[
\limsup_{t\to \infty} \frac{ \sharp \Gamma_1 \cap F_t
}{\lambda(F_t)} \geq \limsup_{t\to \infty} \frac{ \sharp S \cap F_t
}{\lambda(F_t)} = \nu_S > \frac{ \nu_U}{ 2}
   \]
    but at the same time
   \[
\liminf_{t\to \infty} \frac{\sharp \Gamma_0 \cap F_t}{\lambda(F_t)}
\leq \liminf_{t\to\infty} \frac{\sharp (U \setminus S) \cap
F_t}{\lambda (F_t)} \leq \nu_U - \nu_S  < \nu_U / 2.
 \]
  This contradicts the existence of uniform patch frequencies discussed above  and thus excludes
  unique ergodicity.

\smallskip

To show the statement for the case of a fullshift in a fibre, note
that for an embedded fullshift in a fibre the grid $U$ is contained
in $\Oplam (W + \vartheta) - s$ for some $\vartheta\in H$ and
$s\in\R^N$ (compare Remark \ref{rem:es} (d)). By uniform
distribution, Theorem \ref{t:uniform-distribution} (b) , we then
have $\nu_U \leq |W+\vartheta|/\vol(\cL) =|W|/\vol(\cL)$. Now, the
statement follows from the considerations in the first part of the
proof.
\end{proof}

 \section{Random windows and positive entropy}
  \label{section:ProbabilisticallyWindowsAndPositiveEntropy}
In this section we will provide a proof of the main theorem,
Theorem~\ref{t.random_windows}, presented in the introduction. In fact, we will
provide a strengthening of that result. Up to here our discussion involved
fairly general CPS. In this section, we will restrict attention to Euclidean
CPS. More specifically, we will consider the following situation (S):

\begin{itemize}
\item $(\R^N,\R,\cL)$ is a Euclidean CPS with $\cL=A(\Z^{N+1})$, where
  $A=\mathrm{GL}(N+1,\R)$ is such that $\pi_1:\R^N\times\R \to \R^N$ is
  injective on $\cL$ and $\pi_2:\R^N\times\R\to \R$ maps $\cL$ to a dense set
  $L^*=\pi_2(\cL)\ssq \R$.
\item $C\ssq \R$ is a Cantor set of positive measure in $[0,1]$.  Let
  $(G_n)_{n\in\N}$ a numbering of the bounded connected components in $\R \smin
  C$.
\end{itemize}
 We then  define for $\omega \in \Sigma^+=\{0,1\}^\N$ the
  set
  \begin{equation} \label{e.random_window}
  W(\omega) \ = \ C\cup \bigcup_{n:\omega_n=1} G_n.
  \end{equation}
Let $\Proj$ be the   Bernoulli distribution on $\Sigma^+$ with
probability $p\in (0,1)$ (i.e. $\Proj$ is the product measure
$\prod_{n\in \N} \mu$, where $\mu$ is the measure  on $\{0,1\}$
which assigns the value $p$ to $\{0\}$ and $1-p$ to $\{1\}$).

  \begin{lemma}\label{lem:windowproper}
    For $\mathbb{P}$-almost every $\omega \in \Sigma^+$, the window $W(\omega)$
    is proper.
  \end{lemma}
  \begin{proof}
    First, the complement of $W(\omega)$ in $\R$ consists of a union of
    connected components of $\R\smin C$ (those $G_n$ with $\omega_n=0$ and the
    two unbounded components in the complement of $C$). Since these are all
    open, $W(\omega)$ is compact. Further, as $\bigcup_{n:\omega_n=1}G_n$
    belongs to the interior of $W(\omega)$, we have
    $$\partial W(\omega)\ssq C.$$ Next, we are going to show the reverse
    inclusion (almost surely).  Suppose $x \in C$. Since $C$ is perfect, there
    exists a sequence of gaps $\{ G_{n_k} \}_{k \in \mathbb{N}}$ such that $\inf
    G_{n_k} \rightarrow x$ for $k \rightarrow \infty$.  By definition of the
    window, only intervals $G_{n_k}$ with $a_{n_k}(\omega)=1$ are included in
    $W(\omega)$. Since all random variables are independent, the
    Borel-Cantelli-Lemma implies 
    \[\begin{split}
      & \mathbb{P}(\{ \text{ for infinitely many } k,\ G_{n_k} \text{ is
      included in } W(\omega)\}) \ = \ 1 \ , \\ & \mathbb{P}(\{\text{ for
      infinitely many } k,\ G_{n_k} \text{ is not included in } W(\omega) \})
    \ = \ 1\ .
    \end{split}
    \]  Thus, for $\mathbb{P}$-almost every
    $\omega$ there exist subsequences $G_{n_{k_j}}\ssq W(\omega)$ and
    $G_{n_{k'_j}}\ssq H\smin W(\omega)$ such that $\lim_{j\to\infty} \inf
    G_{n_{k_j}} = \lim_{j\to\infty} \inf G_{n_{k'_j}} = x$. Hence, we have $x
    \in \partial W(\omega)$ $\mathbb{P}$-almost surely for every fixed $x\in C$.
    Now, let $M \ssq C$ be a countable and dense subset of $C$. Then for any $x
    \in M$ the argument above shows $x \in \partial W(\omega)$
    $\mathbb{P}$-almost surely. Hence, the countable set $M$ is contained in
    $\partial W(\omega)$ $\Proj$-almost surely. Consequently, we also have
    $$\closure(M)=C\ssq \partial(W(\omega))$$
     $\Proj$-almost surely. Together with the converse inclusion shown above,
    this yields $$\partial(W(\omega))=C$$
    $\Proj$-almost surely.
    From this we obtain
    $$\interior (W(\omega)) = W(\omega) \setminus \partial (W (\omega)) =
    \bigcup_{n:\omega_n=1}G_n$$
    $\Proj$-almost surely. Using this equality and going again  through the argument giving $C\ssq \partial
    W(\omega)$,
we then     find $\Proj$-almost surely
    $$C\ssq \partial (\interior (W(\omega)).$$
From this we then obtain
    $$\closure(\interior(W(\omega))) = \interior(W(\omega))
   \cup \partial (\interior( W(\omega)))\supset  \interior(W(\omega)) \cup C = W(\omega)$$
   and hence
   $$\closure(\interior(W(\omega))) = W(\omega)$$
   for $\mathbb{P}$-almost all $\omega \in X$.
  \end{proof}

  In the next step, we need to find a suitable $h\in \R$ and a respective subset
  $S$ of $\Oplam(C+h)$ of positive asymptotic density.  In order to avoid some
  technicalities later, it turns out convenient to work with $\tilde C=C\smin
  \left(\bigcup_{n\in\N} \partial G_n\cup\{\inf C,\sup C\}\right)$. Note that
  $|\tilde C|=|C|$, since the difference $C\smin \tilde C$ is just the countable
  set of endpoints of the intervals $G_n$ together with the two extremal points
  of $C$.

    \begin{lemma} \label{l.sequence}
       For Lebesgue-almost all $h\in\R$, the sequence $\Oplam ( \tilde C + h)$
       has asymptotic density given by $|C| / |\det A|$.
    \end{lemma}
    \begin{proof} This is a direct consequence of uniform
    distribution, Theorem \ref{t:uniform-distribution}. Note that in
    the case at hand the measure of a fundamental domain is just
    given by $|\det A|$.
%      As in the proof of Lemma~\ref{lem:localMetricIndependence}, by going over
%      to suitable generators $\twovector{a}{b}$ and $\twovector{c}{d}$ of $\cL$,
%      we may assume that $0<c<d=1$ and $0<b<a$. Further, we may assume that
%      $\diam(W)< 1$ and hence, replacing $W$ by some translate if necessary,
%      $W\ssq [0,1)$. Then, similar as in (\ref{e.rotation_representation1}), we
%        obtain
%      \begin{equation}
%       L^*\cap \tilde C+t \ \ssq \ L^*\cap [t,t+1) \ = \ t+\{nc-t\bmod 1\mid n\in\Z\}
%      \end{equation}
%      and thus, writing $l^*(n)=nc-m(n)\in [t,t+1)$ and identifying $[t,t+1)$
%          with $\kreis$ via $p:[t,t+1)\to\kreis,\ x\mapsto x-t\bmod 1$, we
%            further get
%      \begin{equation}
%         \begin{split}
%       \lefteqn{ \nLim 1/n\cdot\sharp\{0\leq j < n\mid l^*(j)\in \tilde C+t\}
%%       }\\ & = \ \nLim 1/n\cdot\sharp\{0\leq j < n\mid R_c^j(-t)\in p(\tilde
 %      C+t)\} \ .
 %        \end{split}
 %     \end{equation}
 %       Thus, if $\kfolge{\tilde l_{k}}=l(\tilde n_k)$ is the subsequence of
 %       those $l(n)$ such that $l_k^*$ is contained in $C$, then this
 %       subsequence has density $\geq \Oplam(C)/a$. (Recall here that the
 %       maximal distance between $(j+1)a-m(j+1)b$ and $ja-m(j)b$ is $a$).
      \end{proof}

  It remains to prove that the random window $W(\omega)$ is $\Proj$-almost
  surely locally topologically independent (as defined in
  Section~\ref{LocalTopInd}) in the sequence $L^*\cap (\tilde C+h)$.

  \begin{lemma}
   \label{lem:independenceInFiniteVariables}
Let $C$ be a Cantor set with positive measure and let $W(\omega)$ be
defined    as in (\ref{e.random_window}). Choose $h\in\R$.  Then for
$\mathbb{P}$-almost every $\omega$ the window $W(\omega)+h$ is
locally topologically independent in $ L^* \cap (\tilde C  +h)$.
  \end{lemma}

  \begin{proof} Let $F$ be an arbitrary finite subset of $L^* \cap
    (\tilde C + h)$.
% Let $\mathcal{I} \ssq \mathbb{N}$ be a finite set with $K =
%\sharp \mathcal{I}$
%      \mathcal{I} \geq 2$ and define the finite subfamily $\mathcal{S} = \{x_1,
%   \dots, x_K \} = \left\{\tilde l^*_{k_i} \mid i \in \mathcal{I} \right\}
%   \ssq L^*$. Without loss of generality, assume $x_1 = 0$ and let
Let
$$ \delta_1\ =\ \frac{1}{2} \cdot \min_{x\neq y\in F} |x-y| \ . $$ Since any
Cantor set is nowhere dense and perfect, there exist gaps $I_1^x\ssq
(0,\delta_1)$ of $C+h-x$, $x\in F$, such that
   \[
    \bigcap_{x\in F} I_1^x \ \neq \ \emptyset\ .
   \]
   By the choice of $\delta_1$, we have $I^x_1+x\neq I^{y}_1+y$ if $x\neq y\in
   F$.  Further, if we let $\delta_2 = \min \{1, \min_{x \in F } \left(\inf
   I_1^x\right) \}$, then by the same argument there exist gaps $I_2^x\ssq
   (0,\delta_2)$ of $C+h-x$ such that
   \[
   \bigcap_{x \in F} I_2^x \ \neq\ \emptyset\ 
   \]
   and $I^x_2+x\neq I^y_2+y$ for $x\neq y\in F$. Proceeding inductively with
   this construction, in the $(n+1)$-st step we define
  \begin{equation*}
   \label{eq:delta}
   \delta_{n+1} = \min \left\{ \frac{1}{n}, \min_{x \in F} \left(\inf I_n^x\right) \right\}
  \end{equation*}
  and choose gaps $I^x_{n+1}\ssq (0,\delta_{n+1})$ of $C+h-x$ such that
  \[
   \bigcap_{i \in \mathcal{C}} I^x_{n+1} \ \neq \ \emptyset \ 
  \]
  and $I^x_{n+1}+x\neq I^y_{n+1}+y$ whenever $x\neq y\in F$. Now, let
  $\nfolge{G_n}$ be a labeling of all gaps of $C+h$. Then by construction, we
  have $I^x_j=G_{n^x_j}-x$ for some $n^x_j\in\N$. Moreover, the choice of the
  $\delta_n$ and $I^x_n\ssq (0,\delta_n)$ ensures that $n^x_j\neq n^{x'}_{j'}$
  if $(x,j) \neq (x',j')$. In particular, this means that
  $(\omega_{n^x_j})^{x\in F}_{j\in\N}$ is a two-parameter family of identically
  distributed independent random variables. Therefore, we obtain that for any
  $a\in\{0,1\}^F$ the set
 \[
   \Omega(a) \ = \ \{ \omega\in\Sigma^+\mid \exists \textrm{ infinitely many }
   j\in\N:\ \omega_{n^x_j}=1 \textrm{ iff } a_x=1  \}
 \]
 has full measure $\Proj(\Omega(a))=1$. However, for all $\omega\in\Omega(a)$,
 we have that
 \[
     I_j\ = \ \bigcap_{x\in F} I^x_j \ \ssq \ \left(\bigcap_{x\in F:
       a_x=1} W(\omega)+h-x\right) \smin \left(\bigcup_{x\in F: a_x =
       0}
     W(\omega)+h-x \right) \ .
\]
Since the intervals $I_j$ are all open and $\lim_{j\to\infty}\inf I_j=0$, this
shows the local topological independence of $W(\omega)$ in $F$. As this works
for any finite subfamily $F$ of $L^*\cap (\tilde C+h)$ and there exist only
countably many such subfamilies, we obtain local topological independence of
$W(\omega)+h$ in $L^*\cap (\tilde C+h) $ for $\Proj$-almost every $\omega\in\Sigma^+$.
 \end{proof}

% Moveover, we   obtain the following additional piece of  information.

%  \begin{lemma}\label{lem:failureue}
%   Let $(\R,\R,\cL)$ be a planar CPS and $W(\omega))$ be defined as in
%   (\ref{e.random_window}). If $|C| > \frac{1}{2}$, then for
%   $\lambda$-almost every $t\in\R$ the DDS $(\Omega(\Oplam(W(\omega)+t)),
%   \mathbb{R})$ is not uniquely ergodic.
%  \end{lemma}
%  \begin{proof} By

%  \end{proof}

We can now summarize the preceding considerations in the following
theorem.

\begin{theorem} \label{t:random_window-extended}
Assume the situation (S) described at the beginning of this section.
Then, there exists a subset $\Sigma_0^+$ of $\Sigma^+$ of full
$\Proj$-measure such that the following holds:

\begin{enumerate}

\item[(a)] For all $\omega \in \Sigma_0^+$ and $\vartheta \in\R$ there exists a
  set $\Xi(\omega)\ssq \T$ of full measure such that the Delone dynamical
  system $(\Omega(\Oplam(W(\omega)+\vartheta)),\R)$ contains an embedded
  fullshift in $\beta_{\vartheta}^{-1}(\xi)$ for every $\xi \in \Xi(\omega)$.

\item[(b)] For all $\omega \in \Sigma_0^+$ and
$\vartheta \in\R$ the Delone dynamical system
$(\Omega(\Oplam(W(\omega)+\vartheta)),\R)$
 has positive topological entropy $\htop (\varphi) = \frac{|C|
\log 2}{|\det A|}$.

\item[(c)]  For every $\omega \in \Sigma_0^+$  there exists a residual set
 $\Theta$ in $\R$ such that the Delone dynamical system
$(\Omega(\Oplam(W(\omega)+\vartheta)),\R)$ is minimal for every
$\vartheta \in \Theta$.

\item[(d)] For every $\omega \in \Sigma_0^+$ and every $\vartheta \in\R$ the
  Delone dynamical system $(\Omega(\Oplam(W(\omega)+\vartheta)),\R)$ is not
  uniquely ergodic provided $C$ additionally satisfies $|C|> 1/2$.
\end{enumerate}
\end{theorem}
\begin{proof} By Lemma \ref{lem:windowproper} there exists a set $\Sigma^+_1$ of full measure in
$\Sigma^+$ such that  $W(\omega)$ is proper for  every
$\omega\in\Sigma^+_1$. By  Lemma
\ref{lem:independenceInFiniteVariables} and Fubini's theorem, there
exists a set $\Sigma^+_2$ of full measure in $\Sigma^+$ such that
for every $\omega \in\Sigma^+_2$ the window  $W(\omega)+h$ is
locally topologically independent in $ L^* \cap (\tilde C  +h)$ for
almost every $h\in  \R$. Set $\Sigma^+_0 :=\Sigma^+_1 \cap
\Sigma^+_2$. Now, consider an arbitrary $\omega \in \Sigma^+_0$.\smallskip

As due to Lemma \ref{l.sequence} the set $\Oplam(\tilde C+h)$ has asymptotic
density $|C| / |\det A|$ for almost every $h\in\R$, we obtain that for almost
every $h\in \R$ the assumptions of Lemma \ref{lem:localTopologicalIndependence}
are satisfied for $W = W(\omega) + h$ and $S= \Oplam (\tilde C + h)$. Therefore,
we then obtain a full measure set $\Xi'(\omega)\ssq\R$ such that for any
$h\in\Xi'(\omega)$ and any $\vartheta\in\R$ there exists an embedded fullshift
in the fibre $\beta^{-1}_\vartheta(\xi)$ for $\xi = [0,h
  -\vartheta]_{\cL}$. Since the existence of an embedded subshift in a fibre is
a property that is invariant under translation by $t$, we then obtain an
embedded fullshift for all $[t,h -\vartheta]_\cL$ with
$(t,h)\in\R^N\times\Xi'(\omega)$. The projection of the latter set to $\T$ gives
the required full measure set $\Xi(\omega)$ that satisfies the assertion (a).

\smallskip

As for (b)   we note that the proven part (a) together with Lemma
\ref{lem:esimpliespositiveh} directly gives  $\htop \geq \frac{|C|
\log 2}{|\det A|}$. On the other hand by the general results of
\cite{HR}  we know that $\htop \leq \frac{|C| \log 2}{\vol(\cL)}$.
Combining these inequalities and using $\vol (\cL) = |\det A|$,  we
arrive at the statement (b).

\smallskip

Statement (c) then follows from general well-known theory. In fact, it is a
direct consequence of Lemma \ref{lem:properandgeneric} combined with (b) of
Lemma \ref{lem:genericrepetitivity} and (b) of Lemma
\ref{l.basic_dynamical_properties}.

\smallskip

 Finally, it remains to show  (d).  The preceding considerations
give almost surely  an embedded fullshift with set of free points
$S$ satisfying $\nu_S = \frac{|C|}{\vol (\cL)}$. Clearly,  the grid
$U$ must be  contained in    $\Oplam([0,1] + h)- t$ for some
$h\in\R$ and $t\in \R$  and hence satisfies
$$\nu_U \leq \nu_{\Oplam([0,1] + h ) - t} \leq
\frac{1}{\vol(\cL)}\ ,$$ where the last inequality follows by uniform
distribution (Theorem \ref{t:uniform-distribution}). This shows
$$\nu_S\  >\  \frac{\nu_U}{2} \ , $$
and Proposition
\ref{prop:UniqueErgodicityDependingOnRelativeDensity} gives the
desired statement.
\end{proof}

\begin{remark} Whenever $W(\omega)$ is proper,  the
dynamical system $(\Omega(\Oplam (W(\omega) +\vartheta)), \R)$ has
the torus $\T$ as its maximal equicontinuous factor and a relatively
dense set of continuous eigenvalues for any $\vartheta \in \R$,
compare Remark \ref{rem:continuousef}.
\end{remark}

 \section{A deterministic construction}
  \label{section:WindowsConsistingOfCantorSets}

In order to prepare for the construction of weak model sets with positive
entropy in the next section, we first provide a deterministic construction of
proper model sets with positive entropy. The starting point of our construction
will be the construction of an initial Cantor set $C_0$ that is adapted to the
respective CPS. To that end, we need to introduce some further notation.

We assume without loss of generality that the matrix $A\in\mathrm{GL}(N+1,\R)$
that defines the lattice $\cL=A(\Z^{N+1})$ is of the form
$A=(a_{i,j})_{i,j=1}^{N+1}$, where $a_{N+1,j}\in (0,1)$ for all $j=1\ld N$ and
$a_{N+1,N+1}=1$. In this case, given any $v=(v_1\ld v_N)\in\Z^N$, there exists a
unique $v_{N+1}\in\Z$ such that
\[
l^*_v \ := \ \pi_2\left(A\cdot
\twovector{v}{v_{N+1}}\right)\ =\ \sum_{j=1}^{N+1} a_{N+1,j}v_j \ \in
\ [0,1)\cap L^* \ .
\]
Note that thus $l^*_v=\sum_{j=1}^{N} a_{N+1,j}v_j \bmod 1$.  Given $v\in\Z^{N}$,
let $\|v\|_\infty=\max_{j=1}^{N} |v_j|$ and fix a numbering $(v(n))_{n\in\N}$ of
$\Z^{N}$ such that $\|v(n)\|_\infty$ is non-decreasing in $n$. Note that this
implies that if we let $\mathcal{N}_t=\Z^N\cap F_t$ and $R_t=\{1\ld (2t+1)^N\}$ for
$t\in\N$, then $v(R_t)=\left\{v(n)\mid n \in R_t\right\} = \mathcal{N}_t$ for all
$t\in\N$.

  \begin{lemma}
   \label{lem:windowsCantorSetsSequence}
    There exists an increasing sequence $(n_k)_{k \in \mathbb{N}} \ssq
    \mathbb{N}$ and a sequence $\{\varepsilon_k\}_{k \in \mathbb{N}} \ssq
    \mathbb{R}_{>0}$ such that the open intervals $I_k = (l^*_{v(n_k)},
    l^*_{v(n_k)}+\varepsilon_k)$ satisfy
   \begin{enumerate}[(i)]
    \item $I_j \cap I_k = \emptyset$ for all $j\neq k$,
    \item $\closure\left(\bigcup_{k \in \mathbb{N}} I_k\right) = [0,1]$,
    \item $\lim_{k \rightarrow \infty} \frac{k}{n_k} > 1/2$.
   \end{enumerate}
  \end{lemma}
  \begin{proof}
   For simplicity, we work in the additive group $\mathbb{R} / \mathbb{Z}$ and
   omit to write $\bmod 1$. In other words, by slighly abusing notation we
   automatically interpret real numbers as elements of the circle. In
   particular, we denote by $d(x,0)$ the distance of $x\in\R$ to the nearest
   integer.  We choose a strictly increasing sequence of integers
   $(\kappa(t))_{t\in\N}$ that satisfies
  \begin{equation}\label{e.kappa_choice}
     \sum_{t\in\N} \frac{\sharp \mathcal{N}_{2t}}{\sharp \mathcal{N}_{\kappa(t)}} \ \leq
     \ \frac{1}{2\cdot5^N} \
  \end{equation}
  and let
  \[
      \eta_{v} \ = \ \min\left\{ d(l^*_u,0)/2 \mid u\in \mathcal{N}_{\kappa(\|v\|_\infty)} \cap
      \Z^{N}\right\}
  \]
   and $J_n := \left[l^*_{v(n)},l^*_{v(n)}+\eta_{v(n)}\right)\cap [0,1]$. Then,
     we define
    $$ B := \left\{n \in \mathbb{N} \mid J_n \cap J_j \neq \emptyset \text{ for some
     } j<n \right\} \ . $$ 

We now want to estimate the cardinality of $B\cap R_t$. To that end, note that
if $J_n\cap J_j\neq \emptyset$ and $J_{n'}\cap J_j\neq \emptyset$ for some
$n,n'>j$, then $$d(l^*_{v(n)},l^*_{v(n')}) \ = \ d(l^*_{v(n)-v(n')},0)\ <
\ 2\eta_{v(j)}$$ and therefore $v(n)-v(n')\notin
\mathcal{N}_{\kappa(\|v(j)\|_\infty)}$. Similarly, $v(n)-v(j)\notin
\mathcal{N}_{\kappa(\|v(j)\|_\infty)}$, and the same for $v(n')-v(j)$. Covering $\mathcal{N}_t\smin
\left(\mathcal{N}_{\kappa(\|v(j)\|_\infty)}+v(j)\right)$ by at most $\sharp \mathcal{N}_{2t}/\sharp
\mathcal{N}_{\kappa(\|v(j)\|_\infty)}$ translates of $\mathcal{N}_{\kappa(\|v(j)\|_\infty)}$ for
each $j$ leads to the following rough estimate. 
   \begin{equation}\label{e.Cantor_construction}
     \begin{split}
     \sharp(B \cap R_t) \ & \leq \ \sum_{j=1}^{\sharp \mathcal{N}_t} \sharp \left\{n \in
     \left\{j+1, \dots, \sharp \mathcal{N}_t\right\} \mid J_n\cap J_j\neq\emptyset
     \right\} \\ &\leq \ \sum_{j=1}^{\sharp \mathcal{N}_t} \frac{\sharp \mathcal{N}_{2t}}{\sharp
       \mathcal{N}_{\kappa(\|v(j)\|_\infty)}} \ \leq \ \sharp \mathcal{N}_{2t} \sum_{k=1}^t
     \frac{\sharp \mathcal{N}_k}{\sharp \mathcal{N}_{\kappa(k)}}
     \ \stackrel{(\ref{e.kappa_choice})}{\leq} \ \frac{\sharp \mathcal{N}_{2t}}{2\cdot
       5^N} \ \leq \ \frac{\sharp \mathcal{N}_t}{2} \ .
     \end{split}
   \end{equation}
   Now let $n_1 = 0$ and define
    $$n_{k+1} = \min \{n > n_k \mid J_n \cap J_{n_j} = \emptyset \text{ for all
   } j \leq k \}. $$ Then by defining $\varepsilon_k =
   \min\{\eta_{v(n_k)},1-l^*_{v(n_k)}\}$ and thus $I_k = J_{n_k}$, property (i)
   follows by construction. Likewise, it is clear that the union of the $I_k$ is
   dense in $[0,1]$. Otherwise, there would be some interval $(a,b)\ssq[0,1]$
   which does not intersect any of the $I_k$. However, in this case any interval
   $J_n$ that is contained in $(a,b)$ would have to appear as some $I_k$ by the
   above construction, leading to a contradiction. Note here that the image of
   $\Z^N$ under the mapping $v\mapsto l^*_v$ is dense in $[0,1]$, so that
   eventually one of the intervals $J_n$ needs to be contained in $(a,b)$.

   Further, we have that
    $$\widehat{B} \ := \ \mathbb{N} \setminus \{n_k \mid k \in \mathbb{N} \}
   \ \ssq \ B. $$ 
   Hence, if we let $\cN=\{n_k\mid k\in\N\}$, then this implies that
  \[
    \sharp\left(\cN\cap R_t\right) \ \geq \ \sharp R_t-\sharp(R_t\cap B) \ \geq
    \ \sharp R_t/2 \ .
  \] 
  If we use in addition that $\lim_{t\to\infty} \sharp(R_t\smin R_{t-1})/\sharp
  R_t=\lim_{t\to\infty} \sharp(\mathcal{N}_t\smin \mathcal{N}_{t-1})/\sharp \mathcal{N}_t = 0$, this yields
  that 
  \[
     \lim_{m\to\infty} \sharp(\cN\cap \{1\ld m\})/m \ \geq \ 1/2 
 \] which in turn implies property $(iii)$.  
  \end{proof}

  Note that
  \[
   C_0 \ = \ [0,1] \setminus \bigcup_{k \in \mathbb{N}} I_k
  \] is a Cantor
  set, since all intervals $I_k$ are pairwise disjoint and their union is dense
  in the circle. It should also be pointed out that property (iii) of the
  preceding lemma implies that $C_0$ has positive measure, but we will not make
  explicit use of this fact.

  \begin{lemma}
   \label{lem:windowsCantorSetConstruction}
   Let $C$ be a Cantor set in $[0,1]$ such that $\{0,1\} \ssq C$. Then there
   exists a sequence of open sets $A_j \ssq [0,1]$ such that
   \begin{enumerate}[(i)]
    \item for all $j \in \mathbb{N}$ the set $A_j$ is a union of gaps of $C$,
    \item $\partial A_j = C$ for all $j \in \mathbb{N}$,
    \item the family $\jfolge{A_j}$ is locally topologically independent in $0$.
   \end{enumerate}
  \end{lemma}
  \begin{proof}
   For any two Cantor sets $C, C' \ssq [0,1]$ with $\{0,1\} \ssq C
   \cap C'$ exists an orientation-preserving homeomorphism of $[0,1]$ which maps
   $C$ to $C'$. So without loss of generality, we may assume $C$ that is the
   middle third Cantor set.  Then we can write
    $$ C = \left\{ \left. \sum_{n=1}^\infty 2a_n3^{-n} \right| a \in \{0,1\}^\mathbb{N}
   \right\}. $$ Let $\mathcal{A} = \bigcup_{n \in \mathbb{N}}\{0,1\}^n$ and
   denote by $\vert a \vert$ the length of $a \in \mathcal{A}$. Then
   \begin{equation}\label{e.Cantor_gaps} G_a \  = \  \left( \sum_{n=1}^{\vert a \vert} 2a_n3^{-n}+3^{-n},
   \sum_{n=1}^{\vert a \vert} 2a_n3^{-n} + 2 \cdot 3^{-\vert a \vert }
   \right)
   \end{equation}
   are exactly the gaps of $C$.  We will construct the sets $A_j$ such that they
   all contain
    $$ A \  = \  \bigcup_{a \in \mathcal{A} : \vert a \vert \in 4\mathbb{N}} G_a $$
   but no $G_a$ with $|a| \in 4\mathbb{N}+1$. Since all points of $C$ are
   approximated by gaps of both types, we always have $\partial A_j = C$. Thus,
   properties $(i)$ and $(ii)$ hold.

 Let $a^{(n)} = 0^{2n+1}1 \in \{0,1\}^{2n+2}$. Choose a countable partition
 $(S_j)_{j \in \mathbb{N}}$ of $\mathbb{N}$ into infinite sets. Further, let
 $((M_j, N_j))_{j \in \mathbb{N}}$ be a numbering of all pairs of disjoint
 finite sets of integers. Then let
    \begin{eqnarray*}
   V_j & = & \bigcup_{n \in \mathbb{N}: j \in M_n} S_n \\
   A_j & = & A \cup \bigcup_{l \in V_j} G_{a^{(l)}} \ .
   \end{eqnarray*}
 For any $n \in \mathbb{N}$
   the set $S_n$ is a subset of all $V_j$ with $j \in M_n$ and disjoint from all
   $V_j$ with $j \in N_n$. Thus, the set
    $$\bigcap_{j \in M_n} A_j \setminus \bigcup_{j \in N_n} A_j $$ contains
   $\bigcup_{l \in S_n} G_{a^{(l)}}$. Since $S_n$ is infinite, this shows the local
   topological independence required in condition $(iii)$.
  \end{proof}

  Now let $C_0 = [0,1] \setminus \bigcup_{k \in \mathbb{N}} I_k$ as above and
  define a window $W$ by
  \begin{equation}\label{e.deterministic_window}
    W \ = \ C_0 \cup \bigcup_{k \in \mathbb{N}} \left(I_k \cap \closure(A_k +
    \inf(I_k))\right) \ .
  \end{equation}
  Note that $\inf(I_k) = l^*_{n_k}$ by construction. Due to
   $$ W = \closure \left( \bigcup_{k \in \mathbb{N}} I_k \cap W \right) =
   \closure \left( \bigcup_{k \in \mathbb{N}} I_k \cap \closure(A_k + \inf(I_k))
   \right) = \closure(\interior(W))$$ the window is
   proper.

   \begin{theorem} \label{t.deterministic} Let $(\R^N,\R,\cL)$ be a CPS and $W$ as
     in (\ref{e.deterministic_window}). Suppose
     $\beta_\vartheta:\Omega(\Oplam(W+\vartheta))\to\T$ is the corresponding
     flow morphism from \eqref{eq:DDSFlowMorphism} and \eqref{r.flow_morphism}.
     Further, choose $S := (l_{v(n_k)})$, where $(n_k)_{k \in \mathbb{N}}$ is
     chosen as in Lemma \ref{lem:windowsCantorSetsSequence} and $l_{v(n_k)}$ is
     defined by $(l_{v(n_k)},l^*_{v(n_k)})\in\cL$. Then the following holds:
   \begin{enumerate}[(a)]
   \item For all $\vartheta\in\R$, the pair $(\Xi,S)$ with
     $\Xi=\beta_\vartheta^{-1}([0,-\vartheta]_\cL)$ is an embedded fullshift,
     with the set $U$ from Definition~\ref{def:DSSubshift} given by
     $U=\Lambda(W)$. In particular, $(\Omega(\Oplam(W+\vartheta)),
     \mathbb{R}^N)$ has positive topological entropy for all $\vartheta \in
     \mathbb{R}$.
    \item The system $(\Omega(\Oplam(W+\vartheta)),\mathbb{R}^N)$ is not
      uniquely ergodic.
   \end{enumerate}
 \end{theorem}
  \begin{proof}
   By construction, the local topological independence of $W$ in $S^*$ is
   equivalent to the local topological independence of the sets $\kfolge{A_k}$
   and thus follows from
   Lemma~\ref{lem:windowsCantorSetConstruction}(iii). Hence, by Lemma
   \ref{lem:localTopologicalIndependence}, $\beta^{-1}([0,-\vartheta]_\cL)$
   contains an embedded fullshift. This proves $(a)$.

 To prove statement $(b)$, observe that with the notation introduced before
 Lemma~\ref{lem:windowsCantorSetsSequence} we have that $U':=\{l_v \mid
 v\in\Z^N\}= \Lambda([0,1])$. Further $\nu_{U'}=1/\det(A)$ by
 Theorem~\ref{t:uniform-distribution} (where part (b) is applied to the window
 $(0,1)$ to obtain a lower estimate). As $U\ssq U'$, we have that $\nu_U\leq
 1/\det(A)$. At the same time, it follows directly from
 Lemma~\ref{lem:windowsCantorSetsSequence}(iii) that
 \[
     \nu_S \ \geq \ \nu_{U'}/2 = \frac{1}{2\det(A)} \ \geq \ \nu_U/2 \ .
 \]
 Hence, $(\Omega(\Oplam(W+\vartheta)),\R^N)$ cannot be uniquely ergodic by
 Proposition~\ref{prop:UniqueErgodicityDependingOnRelativeDensity}.
 \end{proof}

 \section{Weak model sets with positive entropy}
  \label{section:WindowsWithEmptyInterior}

  In this section, we will modify the construction of the previous section
  \ref{section:WindowsConsistingOfCantorSets} such that the resulting window $W$
  has an empty interior, but the dynamical system $(\Omega(\Oplam(W)),
  \mathbb{R}^N)$ still has positive topological entropy. Note that in this case
  we are not dealing with Delone sets.

  \begin{lemma}
   \label{lem:setWithoutInteriorConstruction}
   Let $C \ssq [0,1]$ be the middle third Cantor set. Then there exists a
   sequence of sets $A_j \ssq [0,1]$ such that
   \begin{enumerate}[(i)]
    \item $C \ssq \partial A_j$ for all $j \in \mathbb{N}$,
    \item $\interior(A_j) = \emptyset$ for all $j \in \mathbb{N}$,
    \item the family \jfolge{A_j} is locally metrically independent in $0$.
   \end{enumerate}
  \end{lemma}
  \begin{proof}
   We can write
    $$ C = \left\{ \left. \sum_{n=1}^\infty 2a_n3^{-n} \right| a \in \{0,1\}^\mathbb{N}
    \right\}. $$ As before, let $\mathcal{A} = \bigcup_{n \in
      \mathbb{N}}\{0,1\}^n$ and denote by $\vert a \vert$ the length of $a \in
    \mathcal{A}$ and by $G_a$ the gap of $C$ corresponding to $a$.  Let $K$ be
    another Cantor set in $[0,1]$ such that $\{0,1\} \ssq K$, $\vert K \vert >
    0$ and $0 \in \closess(K) = \{x \in \mathbb{R} \mid \vert B_\varepsilon(0)
    \cap K \vert > 0 \text{ for all } \varepsilon > 0 \}$.  We will construct
    the sets $A_j$ such that each set contains $C$ and, to ensure metric
    independence, we insert $K$ into the gaps of $C$.  Thus, let again $a^{(n)} =
    0^{2n+1}1 \in \{0,1\}^{2n+2}$ and choose a countable partition $(S_j)_{j \in
      \mathbb{N}}$ of $\mathbb{N}$ into infinite sets. Further, let $(M_j,
    N_j)_{j \in \mathbb{N}}$ be a numbering of all pairs of disjoint finite sets
    of integers. Then let
    $$V_j := \bigcup_{n \in \mathbb{N}: j \in M_n} S_n $$
   and
    $$A_j := C \cup \bigcup_{l \in V_j}\left( G_{a^{(l)}} \cap (M +
   \inf(G_{a^{(l)}}))\right).$$ Then conditions (i) and (ii) follow again by
   construction. Further, for any $n,j \in \mathbb{N}$ the set $S_n$ is a subset
   of $V_j$ whenever $j \in M_n$ and disjoint from $V_j$ whenever $j \in N_n$.
   Since $S_n$ is infinite, for any $\varepsilon > 0$ there exists $l \in S_n$
   such that $G_{a^{(l)}} \ssq B_\varepsilon(0)$. Since $0 \in \closess(K)$, the
   set $G_{a^{(l)}} \cap (K + \inf G_{a^{(l)}})$ has positive measure. Thus, as
    $$G_{a^{(l)}} \cap (K + \inf (G_{a^{(l)}})) \ssq B_\varepsilon(0) \cap \left(
 \bigcap_{j \in M_n} A_j \setminus \bigcup_{j \in N_n} A_j \right), $$ the set
 on the right has positive measure. Since this holds for all $\eps>0$ and the
 pair $(M_n,N_n)$ was arbitrary, this shows the metric independence of
 family $\jfolge{A_j}$.
  \end{proof}

  Now, let $(n_k)_{k \in \mathbb{N}}$ and the intervals $I_k$ be as in Lemma
  \ref{lem:windowsCantorSetsSequence}. As in the previous section, let $C_0 =
      [0,1] \setminus \bigcup_{k \in \mathbb{N}} I_k$. Define a window $W$ of
      empty interior by
  \begin{equation} \label{e.meager_window} W \ = \ C_0 \cup \bigcup_{k \in
      \mathbb{N}}\left( I_k \cap (\inf(I_k) + A_k)\right) \ .
  \end{equation}
  Note, that we have $\inf(I_k) = l^*_{v(n_k)}$ by construction of the $I_k$.

  \begin{theorem} \label{t.weak} Let $(\R^N,\R,\cL)$ be a CPS and $W$ as in
    (\ref{e.meager_window}).  Further, choose $S = (l_{n_k})$ as in
    Theorem~\ref{t.deterministic}.

Then for almost all $\vartheta \in \mathbb{R}$ the hull
$\Omega(\Oplam(W+\vartheta))$ contains an embedded fullshift and
$(\Omega(\Oplam(W+\vartheta)),\mathbb{R})$ has positive topological
entropy.
 \end{theorem}
Note that as the window has empty interior in this case, the hull
$\Omega(\Lambda+\vartheta)$ contains the empty set and therefore the fact that
the action cannot be uniquely ergodic is obvious.
 \begin{proof}
   By construction, the metric independence of $W$ in $S^*$ is equivalent to the
   metric independence of the sets $\kfolge{A_k}$ and thus follows from Lemma
   \ref{lem:setWithoutInteriorConstruction}(iii). Hence, by Lemma
   \ref{lem:localMetricIndependence}, $\Omega(\Oplam(W+\vartheta)$ contains an
   embedded fullshift for almost all $\vartheta \in \mathbb{R}$ (compare the
   proof of Theorem~\ref{t.deterministic}).
 \end{proof}

 \begin{remark}
   \label{r.weak_model_fibres}  Similar as in
   Lemma~\ref{lem:localTopologicalIndependence}, one may show that the embedded
   fullshift $\Xi$ which is obtained is contained in $\Oplam(W+\vartheta)$ (that
   is, $\Gamma\ssq\Oplam(W+\vartheta)$ for all $\Gamma\in\Xi$). In the case of
   proper model sets, this was used further to conclude that $\Xi$ is contained
   in the fibre $\beta^{-1}([\vartheta,0]_\cL)$. However, for weak model sets
   there is not analogous statement to that, since a torus parametrisation does
   not exist in this case.
 \end{remark}

\section{Remarks on higher-dimensional internal groups} \label{HigherDimInternal}

In the previous sections, we have concentrated on examples of positive entropy
model sets with one-dimensional internal group $H=\R$. While this makes the
constructions easier on a technical level and allows to avoid heavy notation, it
is also possible to produce similar examples with higher-dimensional internal
groups. There is also a certain motivation for this. The eigenvalues of the
continuous dynamical eigenfunctions in Remark~\ref{rem:continuousef} are those
of the underlying Kronecker flow on the torus $(G\times H)/\cL$. Hence,
increasing the dimension of the internal group leads to a richer spectrum of
continuous eigenfunctions, while keeping the dimension of the direct space
constant.

The analysis in the case of a one-dimensional internal group $H=\R$ is
simplified by the fact that in this situation the boundary of a proper window is
always a Cantor set. In higher dimensions, this boundary also needs to contain
non-trivial connected components, and there is a much greater variety of
possible structures. For this reasons, general statements as the one in
Theorem~\ref{t:random_window-extended} (which starts with an arbitrary Cantor
set) may be more difficult to make. However, when it comes to the construction
of specific examples, the arguments employed in the previous sections can be
adapted with only minor modifications. For instance, a random construction
analogous to that in Theorem~\ref{t:random_window-extended} may be carried out
by starting with a Sierpinski carpet of positive measure, labelling the squares
which were removed in the construction of the carpet and including each of them
in the window independently with probability $1/2$. The proof of
Theorem~\ref{t:random_window-extended} could then easily be modified to show
that the Delone dynamical system on the hull of the resulting model set almost
surely has positive entropy.

Deterministic constructions as in
Sections~\ref{section:WindowsConsistingOfCantorSets} and
\ref{section:WindowsWithEmptyInterior} can equally be carried out with
higher-dimensional $H$. In this case, one would have to start with the
projection of a fundamental domain of the lattice $\cL$ to $H$ and the remove
neighbourhoods of rapidly decreasing size around points in $L^*$ to obtain an
initial Cantor set $C_0$ (compare
Lemma~\ref{lem:windowsCantorSetsSequence}). Pasting in locally topologically
independent sets into these `holes' will then again lead to (weak) model sets
with positive entropy.

 %% ----------------------------------------------------------------------------------------------------------------------------------------------------------------------------------------------------------------
 %% ---------------------------------------------------------------------------------
 %% Literature
 %% -------------------------------------------------------------------------------------------------------------------

% \bibliographystyle{alpha} \bibliography{QuasiBib,dynamics}

\begin{thebibliography}{ABKL15}

\bibitem[Aus88]{Auslander} J.~Auslander.
\newblock{\textit{Minimal flows and their extensions}}.
\newblock North-Holland Mathematical Studies, vol 153, North-Holland
Publishing Co., Amsterdam, 1988.


\bibitem[Auj14]{Auj} J.B.~Aujogue.
\newblock On embedding of repetitive Meyer multiple sets into model multiple sets
\newblock {\em  Ergodic Theory Dyn. Syst.} {\bf 36}(6):1679--1702, 2016.

\bibitem[ABKL15]{ABKL} J.B.~Baptiste, M.~Barge, D.~Lenz,
J.~Kellendonk.
\newblock Equicontinuous factors, proximality and Ellis semigroup for Delone
sets.
\newblock {\em  Mathematics of aperiodic order},  137--194, Prog. Math.,
\textbf{309}, Birkh\"auser/Springer, Basel, 2015.


\bibitem[BG]{BG} M.~Baake, U.~Grimm.
\newblock \textit{Aperiodic Order. Vol.\ $1$: A Mathematical Invitation},
\newblock Cambridge Univ.\ Press, Cambridge, 2013.


\bibitem[BHS16]{BHS} M.~Baake, C.~Huck,  N.~Strungaru.
\newblock{On weak model sets of extremal density}.
\newblock  {\em Indag. Math.} {\bf 28}(1):3--31, 2017.

\bibitem[BJL15]{BJL} M.~Baake, T.~Jaeger, D.~Lenz.  \newblock{Toeplitz flows and
  model sets}.  \newblock {\em Bull. Lond. Math. Soc.} {\bf 48}(4):691--698,
  2016.

\bibitem[BL04]{BL} M.~Baake and D.~Lenz.  \newblock {Dynamical systems on
  translation bounded measures:\ Pure point dynamical and diffraction spectra},
  \newblock {\em Ergod.\ Th.\ \& Dynam.\ Syst.} {\bf 24}(6):1867--93, 2004.


\bibitem[BLM07]{baakeEtAl:CharacterizationOfModelSets} M.~Baake, D.~Lenz, and
  R.V. Moody.  \newblock {Characterization of model sets by dynamical systems}.
  \newblock {\em Ergodic Theory Dyn. Syst.}, {\bf 27}(2):341--382, 2007.

\bibitem[BLR07]{BLR} M.~Baake, D.~Lenz and C.~Richard.
\newblock{Pure point diffraction implies zero entropy for Delone sets with uniform cluster
frequencies}.
\newblock{ {\em Lett. Math. Phys.} {\bf 82}:61--77, 2007.}

\bibitem[BMP00]{BMP} M.~Baake, R.V.~Moody, P.A.B.~Pleasants.  \newblock
  Diffraction from visible lattice points and k-th power free integers.
  \newblock \t, extit{Discr. Math.} \textbf{221}:3--42, 2000.


\bibitem[BM04]{BM} M.~Baake, R.V.~Moody.  \newblock Weighted Dirac combs with
  pure point diffraction.  \newblock \textit{J.\ reine angew.\ Math.\ (Crelle)}
  \textbf{573}:61--94, 2004

\bibitem[Bow71]{bowen:EntropyForGroupEndomorphismsAndHomogeneousSpaces}
  R.~Bowen.  \newblock {Entropy for group endomorphisms and homogeneous spaces}
  \newblock {\em {Trans. Am. Math. Soc.}, \textbf{153}:401--413, 1971.}

\bibitem[Gou04]{Goue}
J.-B.~Gou\'{e}r\'{e}.
\newblock{ Quasicrystals and almost periodicity}.
\newblock{ \textit{Commun.\ Math.\ Phys.}\textbf{255}:651--681, 2005.}

\bibitem[Hof96]{Hof}
A.~Hof.
\newblock On diffraction by aperiodic structures,
\newblock \textit{Commun.\ Math.\ Phys.} \textbf{169}:25--43, 1995.

\bibitem[HB14]{HB} C.~Huck, M.~Baake.
\newblock Dynamical properties of k -free lattice points.
\newblock \textit{Acta Phys. Polon.} A  \textbf{126}:482--485, 2014.

\bibitem[HP13]{HP} C.~Huck, P.~Pleasants.  \newblock{Entropy and diffraction of
  the $k$-free points in $n$-dimensional lattices}.  \newblock{
  \textit{Discr.\ Comput.\ Geom.} \textbf{50}:39--68, 2013.

\bibitem[HR14]{HR}
C.~Huck and C.~Richard. \newblock{On pattern entropy of weak model
sets}. \newblock{ \textit{Discr.\ Comput.\ Geom.} \textbf{54}:741--757, 2015.

\bibitem[KH97]{katok/hasselblatt:1997}
A.~Katok and B.~Hasselblatt.
\newblock {\em Introduction to the Modern Theory of Dynamical Systems}.
\newblock Cambridge University Press, 1997.

\bibitem[KLS15]{KLS} J.~Kellendonk, D.~Lenz, J.~Savinien (eds).
\newblock Mathematics of aperiodic order.
\newblock {\em Progress in
Mathematics}, \textbf{309} Birkh\"auser/Springer, Basel, 2015.

\bibitem[KS14]{KS} J.~Kellendonk, L.~Sadun.
\newblock  {Meyer sets, topological eigenvalues, and Cantor fiber bundles.}
\newblock \textit{ J. Lond. Math. Soc.} (2)  \textbf{89}:114--130, 2014. 

\bibitem[KR15]{KR} G.~Keller, C.~Richard.
\newblock Dynamics on the graph of the torus parametrisation
\newblock  Preprint 2015.  arXiv:1511.06137.

\bibitem[Lag98]{lagarias:GeometricModelsForQuasicrystalsI} J.C. Lagarias.
  \newblock {Geometric Models for Quasicrystals I. Delone Sets of Finite Type.}
  \newblock {\em Discrete Comput. Geom.} {\bf 21}(2):161–191, 1999.

\bibitem[LMS02]{leeEtAl:PurePointDynamicalAndDiffractionSpectra}
J.-Y. Lee, R.V. Moody, and B.~Solomyak.
\newblock {Pure Point Dynamical and Diffraction Spectra}.
\newblock {\em {Annales Henri Poincar\'{e}}} {\bf 3}(5):1003--1018, 2002.

\bibitem[LM16]{LM} D.~Lenz and R.V.~Moody. \newblock{Stationary processes and
  pure point diffraction}. \newblock{\textit{Ergod. Th. \& Dynam.
    Syst.} {\bf 37}(8):2597--2642, 2017.}}

\bibitem[LS03]{LS} D.~Lenz, P.~Stollmann.
\newblock {Delone dynamical systems and associated random operators}.
\newblock  Conference Proceedings, Constanta
(Romania), July 2-7, 2001, J.-M. Combes, J. Cuntz, G.A. Elliott, G.
Nenciu, H. Siedentop, S. Stratila (eds.), Theta Foundation.

\bibitem[LS09]{LStr} D.~Lenz, N.~Strungaru.
\newblock { Pure point spectrum for measure dynamical systems on locally
compact Abelian groups}. \newblock{\textit{J.\ Math.\ Pures Appl.
\textbf{92}:323--341}, 2009.}}

\bibitem[Mey72]{Mey}
Y.~Meyer.
\newblock  {Algebraic Number Theory and Harmonic
Analysis}.
\newblock  North Holland, Amsterdam (1972).

\bibitem[Moo97]{Moo97}
R.V.~Moody.
\newblock Meyer sets and their duals,
\newblock  in:\ R.V.~Moody (ed.),
 \textit{The Mathematics of Long-Range Aperiodic Order},
NATO ASI Series C 489, Kluwer, Dordrecht, 403--441, 1997.

\bibitem[Moo00]{Moo00}
R.V.~Moody. \newblock  Model sets:\ A survey. \newblock  in:\
F.~Axel, F.~D\'{e}noyer and J.P.~Gazeau (eds.) \textit{From
Quasicrystals to More Complex Systems}, Springer, Berlin and EDP
Sciences, Les Ulis, 145--166, 2000.

\bibitem[Moo02]{Moody}
R.~V.~Moody. \newblock{Uniform Distribution in Model Sets.}
\newblock{\textit{ Can. Math. Bull.}, \textbf{ 45}:123--130, 2002}.

\bibitem[RS15]{RS} C.~Richard, N.~Strungaru.
\newblock Pure point diffraction and Poisson summation
\newblock {\em Ann. Henri Poincaré} {\bf 18}(12):3903--3931, 2017.

\bibitem[Sch00]{schlottmann:GeneralizedModelSetsAndDynamicalSystems}
M.~Schlottmann.
\newblock {Generalized Model Sets and Dynamical Systems}.
\newblock {\textit{Directions in Mathematical Quasicrystals}, M.\ Baake and
R.V.\ Moody (eds.), CRM Monograph Series vol.\ 13, AMS, Providence,
RI, 143--159, 2000.}


\bibitem[Stru05]{Stru} N.~Strungaru.
\newblock Almost periodic measures and long range order in Meyer
sets.
\newblock \textit{Disc. Comput. Geom.} \textbf{33}:483--505, 2005.

\bibitem[TZ91]{TZ} A.T. Tagi-Zade. \newblock{ A variational characterization of
  the topological entropy of continuous groups of transformations.  The case of
  {$\mathbb{R}^n$-actions.}  \newblock {\em Mat. Zametki} {\bf 49}:114--123,
  1991. (Translation in Math. Notes 49:305--311, 1991.)}


\bibitem[Wal82]{walters:AnIntroductionToErgodicTheory}
P.~Walters.
\newblock {\em {An Introduction to Ergodic Theory}}.
\newblock {Springer Verlag, 1982.}


\end{thebibliography}

\end{document}